\documentclass[11pt]{article}
\usepackage{amssymb,latexsym,amsmath,amsbsy,amsthm,amsxtra,amsgen,dsfont,pdfsync}
\newfont{\msbm}{msbm10 at 11pt}
\newcommand {\R} {\mbox{\msbm R}}

\newcommand {\N} {\mbox{\msbm N}}

\newcommand {\1} {\mathds{1}}

\newfont{\msbmsm}{msbm10 at 8pt}

\def\eps{\varepsilon}

\newtheorem{Theo}{Theorem}
\newtheorem{Lemma}[Theo]{Lemma}
\newtheorem{Cor}[Theo]{Corollary}
\newtheorem{Prop}[Theo]{Proposition}

\newtheorem{Rmk}[Theo]{Remark}

\begin{document}
\title{The size of the last merger and time reversal in $\Lambda$-coalescents}

\author{G\"otz Kersting\thanks{Supported in part by DFG priority program 1590}\,, Jason Schweinsberg\thanks{Supported in part by NSF Grant DMS-1206195}\,, and Anton Wakolbinger$^\ast$}

\maketitle
\begin{abstract}
We consider the number  of blocks involved in the last merger of a $\Lambda$-coalescent started with $n$ blocks.  We give  conditions under which, as $n \to \infty$, the sequence of these random variables a) is tight, b) converges in distribution to a finite random variable  or c) converges to infinity in probability. Our conditions are optimal for $\Lambda$-coalescents that have a dust component. For general $\Lambda$, we relate the three cases to the existence, uniqueness and non-existence of invariant measures for the dynamics of the block-counting process, and in case b) investigate the time-reversal of the block-counting process back from the time of the last merger.  
\end{abstract}

\section{Introduction and main results}

We consider coalescents with multiple mergers, also known as $\Lambda$-coalescents, which were introduced in 1999 by Pitman \cite{pit99} and Sagitov \cite{sagitov}.  If $\Lambda$ is a finite measure on $[0,1]$, then the $\Lambda$-coalescent started with $n$ blocks is a continuous-time Markov chain $(\Pi_n(t), t \geq 0)$ taking its values in the set of partitions of $\{1, \dots, n\}$.  It has the property that whenever there are $b$ blocks, each possible transition that involves merging $k \geq 2$ of the blocks into a single block happens at rate
\begin{equation}\label{lambk}
\lambda_{b,k} = \int_0^1 p^{k-2} (1-p)^{b-k} \: \Lambda(dp),
\end{equation}
and these are the only possible transitions.  One can also define the $\Lambda$-coalescent started with infinitely many blocks, which is a continuous-time Markov process $(\Pi_{\infty}(t), t \geq 0)$ taking its values in the set of partitions of the positive integers such that for all $n$, the restriction of $(\Pi_{\infty}(t), t \geq 0)$ to the integers $\{1, \dots, n\}$ has the same law as $(\Pi_n(t), t \geq 0)$.

Let $N_n(t)$ be the number of blocks in the partition $\Pi_n(t)$.  Denote by $T_n = \inf\{t: N_n(t) = 1\}$  the time of the last merger.  In this paper, we are interested in the distribution of\[ L_n:=N_n(T_n-),\] the number of blocks that coalesce during the last merger.  
The asymptotic behaviour of the distribution of $L_n$ depends on how much mass the measure $\Lambda$ has in the vicinity of point 1. Here it turns out to be decisive whether or not the finiteness condition
\begin{align}
\int_0^1 |\log(1-p)|\, \Lambda(dp) < \infty 
\label{maincond}
\end{align}
is valid. We shall prove that \eqref{maincond} together with a logarithmic nonlattice property implies convergence on the sequence $(L_n)$ in distribution. Without additional assumptions condition \eqref{maincond} entails tightness of $(L_n)$, but in general not convergence. In the presence of dust, \eqref{maincond} turns out to be necessary for tightness of $(L_n)$. When the $\Lambda$-coalescent comes down from infinity, which means that almost surely $N_{\infty}(t) < \infty$ for all $t > 0$, we have $T_{\infty} < \infty$ almost surely.  See \cite{sch00} for a necessary and sufficient condition for the $\Lambda$-coalescent to come down from infinity.  In this case the distribution of $L_n$ converges as $n \rightarrow \infty$  to the distribution of $N_{\infty}(T_{\infty}-)$.  

A second issue is the characterisation of the limit distribution of $L_n$ in case of convergence by means of invariant measures $\mu$. 
Let
\[ \rho_{ij} := \binom i {i-j+1} \lambda_{i,i-j+1}  , \quad \rho_i:= \sum_{j=1}^{i-1}\rho_{ij}, \qquad 1\le j <i .\]
Then $\rho_{ij}$ is the rate at which $N_n$ jumps from state $i$ to $j$, and $\rho_i$ is the total rate of a jump from~$i$. In particular note that $\rho_{i1}=\lambda_{i,i}$. We consider locally finite measures $\mu=(\mu_i)_{i\ge 2}$ on $\{2,3,\ldots\}$ which fulfill the equations
\begin{align} \sum_{j=i+1}^\infty \mu_j \rho_{ji}=\mu_i \rho_i  , \quad i \ge 2 , \quad \text{and}\quad \sum_{j= 2}^\infty\mu_j\rho_{j1} =1.
\label{quasiinva}
\end{align}
Note that for such  measures $\mu$ we have $\mu_i>0$ for all $i \ge 2$.  The first property in \eqref{quasiinva}  says that the measure $\mu$ on $\{2,3,\ldots\}$  is {\em $\rho$-invariant}:  for each $i\ge2$  the flow of mass into the state $i\ge2$ equals the flow out of $i$. The second property says that the total flow out of the set $\{2,3,\ldots\}$ equals one.  We shall address questions of existence and uniqueness of solutions to \eqref{quasiinva} and shall in particular prove that in case of convergence of $L_n$ the limiting distribution has weights $\mu_i\rho_{i1}$, $i \ge 2$, with $(\mu_i)_{i\ge 2}$ being the unique solution of \eqref{quasiinva}. Moreover this representation of the limit will allow us to identify the time-reversal of the block-counting process.

H\'enard \cite{henard} and M\"ohle \cite{mohle14} were able to calculate the limiting distribution for $L_n$ when $\Lambda$ is the beta distribution with parameters $2 - \alpha$ and $\alpha$ for $0 < \alpha < 2$.  Note that this coalescent process comes down from infinity only when $1 < \alpha < 2$.  Earlier, Goldschmidt and Martin \cite{gm05} had calculated this distribution for the Bolthausen-Sznitman coalescent, which is the case $\alpha = 1$.  Abraham and Delmas found this limit for $\alpha = 1/2$ in \cite{ad13}, and for all $\alpha \in (0, 1/2]$ in \cite{ad15}.

We are now going to present our main results. Throughout, we will assume that $\Lambda$ is a nonzero, finite measure  on $[0,1]$.  Theorem \ref{tight} concerns tightness.  

\begin{Theo}\label{tight}
Suppose that condition \eqref{maincond} is satisfied.
Then the sequence $(L_n)_{n \ge 1}$ is tight. 
\end{Theo}

Under an additional regularity condition, we are able to show that the distribution of the number of blocks involved in the last merger tends to a limit as $n \rightarrow \infty$.  We call the measure $\Lambda$ {\em log-nonlattice} if   
\[ \forall \, d>0:\sum_{z=1}^\infty \Lambda (\{1- e^{ -zd}\}) < \Lambda((0,1]) .\]

\begin{Theo}\label{limitdist}
Suppose (\ref{maincond}) holds, and $\Lambda$ is log-nonlattice.  Then the sequence $(L_n)_{n \ge 1}$ converges in distribution.
\end{Theo}

In this theorem the log-nonlattice assumption cannot be completely avoided. Indeed we shall show below that when $\Lambda$ has all its mass at one single point within $(0,1)$, the sequence $(L_n)_{n \ge 1}$, though tight, does not converge in distribution. It is natural to conjecture that in the lattice case we always will experience such non-convergence.

The next theorem shows that condition \eqref{maincond} is necessary for tightness of the size of the last merger in the presence of dust.

\begin{Theo}\label{nottight}
Suppose 
\begin{equation}\label{dustcond}
\int_0^1 p^{-1} \: \Lambda(dp) < \infty
\end{equation}
and hence in particular $\Lambda(\{0\})=0$. Also suppose
\begin{equation}\label{maincond2}
\int_0^1 |\log(1 - p)| \: \Lambda(dp) = \infty.
\end{equation}
Then for all positive integers $\ell$, we have
\begin{equation}\label{NnKlim}
\lim_{n \rightarrow \infty} P(L_n \leq \ell) = 0.
\end{equation}
\end{Theo}

It was shown in \cite{pit99} that (\ref{dustcond}) is the condition under which the $\Lambda$-coalescent has a dust component, which means that for all $t > 0$, the partition $\Pi_{\infty}(t)$ contains singleton blocks almost surely.  We can see from the statements of Theorems \ref{tight} and \ref{nottight} that when $\Lambda$ satisfies (\ref{dustcond}), the condition (\ref{maincond2}) is necessary and sufficient for (\ref{NnKlim}) to hold.  Therefore, the only case that remains open is the case when the $\Lambda$-coalescent fails to come down from infinity but there is no dust component.  In that case, we expect that it is possible that (\ref{maincond2}) holds but (\ref{NnKlim}) fails to hold.  

The central tool for the proof of Theorem \ref{nottight} is a uniform approximation of  $\log N_n(t)$ by the solution of an SDE driven by a subordinator, see Theorem \ref{couplingTh} in Section 3 and its corollaries. These results can be seen as refinement and generalization of the subordinator approximation by Gnedin, Iksanov, and Marynych \cite{gim11} in the presence of a dust component, see Remark \ref{gimremark} below.

Whenever the random variables $L_n$  converge in distribution,  it is natural to ask whether convergence in distribution holds for the  block-counting processes $N_n = (N_n(t))_{t\ge 0}$ as $n \to \infty$ in any finite observation window around state 1. An appropriate description is by means of time-reversal. As a tool  we use $\rho$-invariant measures satisfying equations \eqref{quasiinva}. 
%
Existence and uniqueness of such measures are closely related to the asymptotic behaviour of the sequence of distributions of the last merger sizes $L_n$. 

\begin{Theo}\label{quasiinv}
\begin{enumerate}
\item[\em (i)]
If $L_n\to \infty$ in probability as $n\to \infty$, then there is no solution to \eqref{quasiinva}. 
\item[\em (ii)]
If there is a probability measure $\pi=(\pi_i)_{i \ge 2}$ on $\{2,3,\ldots\}$ and a sequence of positive numbers $\alpha_n$, $n \ge 1$, not converging to 0, such that as $n \to \infty$
\[ P(L_n=i) \sim \alpha_n \pi_i\]
for all $i\ge 2$, then  the  measure $\mu=(\mu_i)_{i\ge 2}$ given by $\mu_i\rho_{i1}= \pi_i$, $i \ge 2$, is the unique solution to \eqref{quasiinva}.

In particular, if the sequence $(L_n)_{n \ge 1}$ converges in distribution to a finite random variable $L_\infty$, then 
\[  P(L_\infty=i) = \mu_i \rho_{i1}=\mu_i\lambda_{i,i}, \quad i \ge 2.\]
\item[\em (iii)]
In all other cases, there exist at least two different solutions of \eqref{quasiinva}.

In particular we have at least two solutions if the sequence $(L_n)_{n \ge 1}$ is tight, but not convergent in distribution.

\end{enumerate}
\end{Theo}

\noindent
In the case of a coalescent coming down from infinity, as already stated above, item (ii) applies.  In the presence of dust the three cases all occur (see Theorem \ref{limitdist}, Theorem \ref{nottight}, and Section \ref{Eldon}). At first sight one may expect that the condition $P(L_n=i) \sim \alpha_n \pi_i$ in item (ii) will occur only with $\alpha_n \to 1$, that is the random variables $L_n$ converge in distribution.  At the moment, however, we cannot exclude the possibility that the sequence $(\alpha_n)$ is not convergent.

\mbox{}\\
Theorem \ref{quasiinv} will allow us to treat 
the time-reversal  $\hat N_n=(\hat N_n(t))_{t \ge 0}$ of the block-counting process $N_n$. This process is defined  as the c\`adl\`ag process given by
\[ \hat N_n(t):= \begin{cases} N_n((T_n-t)-)  &  \text{for }0 \le t< T_n, \\ n & \text{for } t \ge T_n. \end{cases} \]
In particular we have $\hat N_n(0)=L_n$. 

\begin{Theo}\label{timerev}
If the sequence $(L_n)_{n \geq 1}$ converges in distribution, then also
the sequence of processes $(\hat N_n)_{n \ge 1}$ converges in distribution in Skorohod space.  The limit  $\hat N_\infty$ is  a Markov jump process with values in $\{2,3,\ldots\}$ and jump rates
\[ \hat \rho_{ij} := \frac{ \mu_j\rho_{ji}}{\mu_i}, \quad i<j ,  \]
where the $\mu_i$ are the weights of the $\rho$-invariant measure from Theorem \ref{quasiinv} (ii).
\end{Theo} 

\begin{Rmk} For the Kingman coalescent a direct computation shows that  the solution of \eqref{quasiinva} is given by
\[ \mu_i= \frac 2{i(i-1)} \ , \ i \ge 2. \]
 For $\Lambda = {\rm Beta}(2-\alpha, \alpha)$ with $\alpha \in (0,2)$, H\'enard \cite{henard} and M\"ohle \cite{mohle14} obtain
$$P(L_\infty = i)=\begin{cases}(-1)^{i-1}\alpha\binom{\alpha-1}{i-1}\int_{[0,1]}\frac {x^{i-1}}{1-(1-x)^{1-\alpha}} \, dx \quad &\mbox{ if } \alpha \neq 1\\ -\frac 1{i-1}\int_{[0,1]}\frac{x^{i-1}}{\log(1-x)} \, dx \quad &\mbox{ if } \alpha = 1.\end{cases}$$
Since in this case  $\lambda_{i,i} = \frac{B(i-\alpha, \alpha)}{B(2-\alpha, \alpha)}$, we obtain  from Theorem \ref{quasiinv} an expression for the $\rho$-invariant measure $\mu$ obeying \eqref{quasiinva}.
\end{Rmk}

The rest of this paper is organized as follows.  We prove Theorem \ref{tight} in Section 2.  In Section 3, we show how to approximate the number of blocks in the $\Lambda$-coalescent by means of a subordinator when (\ref{dustcond}) holds.  We prove Theorem \ref{limitdist} in Section 4.  In Section 5 we give an example in which $(L_n)_{n \geq 1}$ is tight but does not converge in distribution because the log-nonlattice assumption in Theorem~\ref{limitdist} fails.  We then derive Theorem \ref{nottight} in Section 6, and we prove Theorems \ref{quasiinv} and \ref{timerev} in Section 7.

\section{Proof of Theorem 1}\label{sec2}

It will be useful throughout the paper to work with a Poisson process construction of the \linebreak $\Lambda$-coalescent.  The construction that we will give is a slight variation of the original such construction provided by Pitman in \cite{pit99}.

Assume 
$\Lambda(\{0\}) = 0$.  Let $\Psi$ be a Poisson point process on $(0, \infty) \times (0, 1] \times [0, 1]^n$ with intensity $$dt \times p^{-2} \Lambda(dp) \times du_1 \times \dots \times du_n.$$  Let $\Pi_n(0) = \{\{1\}, \dots, \{n\}\}$ be the partition of the integers $1, \dots, n$ into singletons.  Suppose $(t, p, u_1, \dots, u_n)$ is a point of $\Psi$, and $\Pi_n(t-)$ consists of the blocks $B_1, \dots, B_b$, ranked in order by their smallest element.  Then $\Pi_n(t)$ is obtained from $\Pi_n(t-)$ by merging together all of the blocks $B_i$ for which $u_i \leq p$ into a single block.  These are the only times that mergers occur.  This construction is well-defined because almost surely for any fixed $t_0 < \infty$, there are only finitely many points $(t, p, u_1, \dots, u_n)$ of $\Psi$ for which $t \leq t_0$ and at least two of $u_1, \dots, u_n$ are less than or equal to $p$.  The resulting process $\Pi_n = (\Pi_n(t), t \geq 0)$ is the $\Lambda$-coalescent.  When $(t, p, u_1, \dots, u_n)$ is a point of $\Psi$, we say that a $p$-merger occurs at time $t$.  

We will need the following simple lemma pertaining to the rate at which the number of blocks decreases.

\begin{Lemma}\label{ESn}
Consider the $\Lambda$-coalescent $\Pi_n$ started with $n$ blocks and let $0<\gamma < 1$.  Let $W_n = \inf\{t\ge 0: N_n(t) \leq \gamma n\}$.  Then there exists a positive constant $C$, depending on $\Lambda$ and $\gamma$ but not on $n$, such that $E[W_n] \leq C$ for all $n \geq 2$.
\end{Lemma}

\begin{proof}
For $2 \leq k \leq n$, the probability that $k$ is the smallest integer in one of the blocks of $\Pi_n(t)$ is bounded above by the probability that the integers $1$ and $k$ do not merge before time $t$, which is $e^{-\lambda_{2,2} t}$.  Therefore, $$E[N_n(t)] \leq 1 + (n-1) e^{-\lambda_{2,2} t}.$$  Thus, using Markov's Inequality, $$P(W_n > t) = P (N_n(t) > \gamma n ) \leq \frac{E[N_n(t)]}{\gamma n} \leq \frac{1}{\gamma n} + \frac{(n-1) e^{-\lambda_{2,2} t}}{\gamma n}.$$  Because $\lambda_{2,2} = \Lambda([0,1]) > 0$ by assumption, there exists $t_0 > 0$ such that $P(W_n > t_0) \leq 1/2$ for sufficiently large $n$.  By increasing the value of $t_0$ if necessary, we can arrange for this inequality to hold for all $n \geq 2$.  Then by repeatedly applying the Markov property, we get $P(W_n > m t_0) \leq 2^{-m}$ for all positive integers $m$.  It follows that $E[W_n] \leq 2 t_0$ for all $n \geq 2$, which gives the result.
\end{proof}

\begin{Lemma}\label{binbound}
Let $B_{b,p}$ have a binomial distribution with parameters $b$ and $p$.  Then for all $k,x > 0$
\begin{equation}\label{bin1}
P(B_{b,p} \geq b - k) \leq 2p^{\lfloor b/2k \rfloor}
\end{equation}
and
\begin{equation}\label{bin2}
 P(B_{b,p} \geq x) \leq p^x2^b .
\end{equation}
Moreover,
\begin{equation}\label{binexp}
E \bigg[ \frac{1}{B_{b,p} + 1} \bigg] = \frac{1 - (1-p)^{b+1}}{(b+1)p}.
\end{equation}
\end{Lemma}

\begin{proof}
To prove (\ref{bin1}), let $\xi_1, \dots, \xi_b$ be independent random variables with $P(\xi_i = 1) = p$ and $P(\xi_i = 0) = 1-p$.  Observe that $$P \bigg( \bigcup_{i=1}^j \{\xi_i = 0\} \bigg| \sum_{i=1}^b \xi_i \geq b - k \bigg) \leq j P \bigg( \xi_1 = 0 \bigg| \sum_{i=1}^b \xi_i \geq b - k \bigg) \leq \frac{jk}{b}.$$  In particular, if $j \leq b/2k$, then the right-hand side is less than $1/2$ and, taking complements, we get $$P \bigg( \xi_1 = \dots = \xi_j = 1 \bigg| \sum_{i=1}^b \xi_i \geq b - k \bigg) \geq \frac{1}{2}.$$  It follows by taking $j = \lfloor b/2k \rfloor$ that $$P \bigg( \sum_{i=1}^b \xi_i \geq b - k \bigg) \leq 2 P(\xi_1 = \dots = \xi_j = 1) = 2 p^{\lfloor b/2k \rfloor},$$ which gives (\ref{bin1}).  

To show \eqref{bin2} we obtain from an exponential Markov inequality that
\begin{equation}
P(B_{b,p} \geq x) \leq e^{-\lambda x} (1+pe^{\lambda})^b 
\end{equation}
with $\lambda >0$. Putting $\lambda = -\log p$ the inequality follows.

Finally, we have
$$E \bigg[ \frac{1}{B_{b,p} + 1} \bigg] = \sum_{k=0}^b \frac{1}{k+1} \binom{b}{k} p^k (1-p)^{b-k} = \frac{1}{(b+1)p} \sum_{k=0}^b \binom{b+1}{k+1} p^{k+1}(1-p)^{b-k},$$
which equals the right-hand side of (\ref{binexp}).
\end{proof}

Theorem \ref{tight} is an immediate consequence of Proposition \ref{mKmprop} below when $m = 1$. (We state this proposition in a more general form, which we will use in the proof of Theorem \ref{limitdist}.)

\begin{Prop}\label{mKmprop}
Suppose that (\ref{maincond}) holds.  Then for all $\eps > 0$, there exists a positive integer $K_{\eps}$ such that $P(m < N_n(t) \leq K_{\eps}m \mbox{ for some }t \geq 0) > 1 - \eps$ for all integers $m$ and $n$ such that $1 \leq m < n$.
\end{Prop}

\begin{proof}
For $K\ge 2$, let $A_{m,n}$ be the complement of the event that $m < N_n(t) \leq Km \mbox{ for some }t \geq 0$.  If $A_{m,n}$ occurs, then for some nonnegative integer $\ell$, a single merger takes the coalescent from between $2^{\ell}Km + 1$ and $2^{\ell+1}Km$ blocks down to $m$ blocks or fewer. 

Suppose there are $b$ blocks in the $\Lambda$-coalescent at some time, where $b \geq 2^{\ell}Km + 1$, and then a $p$-merger occurs.  For the $p$-merger to take the coalescent down to $m$ blocks or fewer, the number of blocks that participate in the merger must be at least $b - m + 1$.  By (\ref{bin1}), if $m \geq 2$, then the probability that this occurs is bounded above by $$2 p^{\lfloor b/2(m-1) \rfloor} \leq 2p^{\lfloor (2^{\ell} Km + 1)/(2(m-1)) \rfloor} \leq 2p^{\lfloor 2^{\ell}(K/2)\rfloor} \leq 2p^{2^{\ell}(K/2) - 1}.$$  If $m = 1$, this probability is bounded above by $p^b \le 2 p^{2^{\ell}(K/2) - 1}$.  Because, from the Poisson process construction of the $\Lambda$-coalescent, we know that $p$-mergers take place at rate $p^{-2} \: \Lambda(dp)$, it follows that the rate of events that take the coalescent down to $m$ blocks or fewer is bounded above by $$2 \int_0^1 p^{2^{\ell}(K/2) - 3} \: \Lambda(dp).$$  By Lemma \ref{ESn}, the expected amount of time for which the number of blocks is between $2^{\ell}Km + 1$ and $2^{\ell+1}Km$ is bounded above by $C$ for all $\ell$.  Therefore,
\begin{align*}
P(A_{m,n})& \leq \sum_{\ell = 0}^{\infty} 2C \int_0^1 p^{2^{\ell}(K/2) - 3} \: \Lambda(dp) \\
&= 2C \int_0^1 \sum_{\ell = 0}^{\infty} p^{2^{\ell}(K/2) - 3} \: \Lambda(dp) \\
&\leq 2C \int_0^1 \sum_{\ell=0}^{\infty} p^{2^{\ell}((K/2) - 3)} \: \Lambda(dp).
\end{align*}
For any $a>0$ and any $x \in (0,1)$, we have $$\sum_{\ell = 0}^{\infty} x^{2^{\ell} a} = x^a + \sum_{\ell = 1}^{\infty} \sum_{j = 2^{\ell - 1}+1}^{2^{\ell}} \frac{x^{2^{\ell} a}}{2^{\ell - 1}} \leq x^a + \sum_{\ell = 1}^{\infty} \sum_{j = 2^{\ell - 1}+1}^{2^{\ell}} \frac{2 x^{ja}}{j} = 2 \sum_{j=1}^{\infty} \frac{x^{ja}}{j} = 2|\log (1 - x^a)|.$$  
Therefore, if $1 \leq m < n$, then for $K > 6$
$$P(A_{m,n}) \leq 4C \int_0^1 |\log (1 - p^{(K/2)-3})| \: \Lambda(dp).$$  
It follows from (\ref{maincond}) and the Dominated Convergence Theorem that this expression tends to zero as $K \rightarrow \infty$, which gives the result.
\end{proof}

\section{An approximation in the case of dust}

Condition \eqref{dustcond} allows us to approximate the number of blocks in the $\Lambda$-coalescent by a subordinator.  For this, we will use the construction of the $\Lambda$-coalescent from the Poisson point process $\Psi$ introduced at the beginning of Section \ref{sec2}.  Let $\phi: (0, \infty) \times (0, 1] \times [0, 1]^n \rightarrow (0, \infty) \times (0, \infty]$ be the function defined by $$\phi(t, p, u_1, \dots, u_n) = (t, -\log(1-p)).$$  Now $\phi(\Psi)$ is a Poisson point process, and we can define a pure jump subordinator $(S(t), t \geq 0)$ having the property that $S(0) = 0$ and, if $(t, x)$ is a point of $\phi(\Psi)$, then $S(t) = S(t-) + x$.  This subordinator first appeared in the work of Pitman \cite{pit99} and was used to approximate the block-counting process by  Gnedin et al. \cite{gim11} and M\"ohle \cite{mohle10}. The next theorem 
provides a refinement.

Define
\begin{align} \label{defb}
f(y):= \int_0^1 \frac{1-(1-p)^{e^y}}{e^y} \frac {\Lambda(dp)}{p^2} , \quad y \in \mathbb R.
\end{align}
From \eqref{dustcond}, we see that $f(y)$ is finite for all $y \in \mathbb R$. Also  $f$  is decreasing with $\lim_{y\to \infty} f(y) = 0$, because for fixed $p$ the integrand has this behaviour. Let $Y_n = (Y_n(t))_{t\ge 0}$ be the solution of the SDE
\begin{align} 
\log n -S(t) = Y_n(t)-\int_0^tf(Y_n(s))ds, \quad t\ge 0. \label{SDE1}
\end{align}
Our goal is to show that for coalescents with dust the   log of the block-counting process  follows closely the process $Y_n$, up to the time when $N_n$ has nearly reached the state 1.  The drift $f(Y_n(t))\, dt$ appears because a merging of $b$ out of $N_n(t)$ lines results in a decrease by $b-1$ and not  by $b$ lines, see equation \eqref{intuition} below. For this purpose, we define for any $k>1$
\begin{align}\label{taukndef}
\tau_{k,n} := \inf\{ t \ge 0:  N_n(t) < k\} .
\end{align}
\begin{Theo} \label{couplingTh}
Under assumption \eqref{dustcond}, for all $\eps > 0$ there is an integer $k\ge 2$ such that for all~$n$,
\begin{align}\label{couplingY}
P \bigg( \sup_{ t \in [0,\tau_{k,n}]\cap[0,T_n) } \big| \log N_n(t) - Y_n(t) \big| \le \eps \bigg) > 1 - \eps.
\end{align}
\end{Theo}
Note that \eqref{couplingY} controls the distance between $Y_n$ and $\log N_n$ up to the first time point when $N_n$ jumps below $k$. This time point is excluded only if the jump leads directly to $1$, i.e. on the event $\{\tau_{k,n} = T_n\}$.

Before proving this theorem let us derive some consequences.
\begin{Cor}
Under assumption \eqref{dustcond}, for all $\eps > 0$ there is an integer $\ell$ such that
\begin{align}\label{couplingY1}
P \bigg( \sup_{0\le t <T_n } \big| \log N_n(t) - Y_n(t) \big| \le \ell \bigg) > 1 - \eps.
\end{align}
\end{Cor}
\begin{proof}
For $\tau_{k,n}< t < T_n$ and $|\log N_n(\tau_{k,n})-Y_n(\tau_{k,n})| \le \eps$ we have, since $f(x) \ge 0$,
\[Y_n(t) \ge S(\tau_{k,n})-S(t)+Y_n(\tau_{k,n}) \ge S(\tau_{k,n})- S(T_n) -\eps. \]
Hence, since $f$ is decreasing, 
\begin{align*}  |Y_n(t)-Y_n(\tau_{k,n})| &\le S(T_n)-S(\tau_{k,n}) + \int_{\tau_{k,n}}^{T_n} f(Y_n(s))\, ds\\
&\le S(T_n)-S(\tau_{k,n}) + f\big(S(\tau_{k,n})- S(T_n) -\eps\big)(T_n-\tau_{k,n}) 
\end{align*}
and therefore
\begin{align*}
|\log N_n(t)- Y_n(t)| &\le |\log N_n(t)- \log N_n(\tau_{k,n})| + |\log N_n(\tau_{k,n})- Y_n(\tau_{k,n})|+ |Y_n(\tau_{k,n})-Y_n(t)| \\
& \le \log k + \eps +  S(T_n)-S(\tau_{k,n}) + f\big(S(\tau_{k,n})- S(T_n) -\eps\big)(T_n-\tau_{k,n}).
\end{align*}
By the strong Markov property, $T_n-\tau_{k,n}$ is stochastically bounded from above by $T_k$ and similarly $S(T_n)-S(\tau_{k,n})$ by $S(T_k)$. Therefore $\sup_{\tau_{k,n}< t < T_n} |\log N_n(t)- Y_n(t)|$ is stochastically bounded on the event $|\log N_n(\tau_{k,n})-Y_n(\tau_{k,n})| \le \eps$. The claim now follows from Theorem \ref{couplingTh}.
\end{proof}

Since $f(x)\to 0$ for $x \to \infty$, the processes $Y_n$ and $\log n-S$ are in view of \eqref{SDE1} close to each other, and one may wonder whether also $\log n-S$ is suitable to approximate the log of the block-counting process. This works under a stronger condition.

\begin{Cor} \label{ApproxS}
Under the assumption 
\begin{align} \int_0^1| \log p| \, \frac {\Lambda (dp)}p < \infty,
\label{Gnd}
\end{align} 
for all $\eps > 0$ there is an integer $k\ge 2$ such that for all~$n$,
\begin{align}\label{couplingY2}
P \bigg( \sup_{ t \in [0,\tau_{k,n}]\cap[0,T_n) } \big| \log N_n(t) -\log n+ S(t) \big| \le \eps \bigg) > 1 - \eps.
\end{align}
\end{Cor} 

\begin{proof}
For $z \ge 1$ we have $1-(1-p)^z \le pz\wedge 1$.  Therefore with $z=e^y$
\begin{align*}
\int_0^\infty f(y)\, dy &= \int_0^1 \int_0^\infty \frac{1-(1-p)^{e^y}}{e^y} \, dy \, \frac{\Lambda(dp)}{p^2}\\ 
&\le \int_0^1 \bigg( \int_0^{|\log p|} p \, dy + \int_{|\log p|}^\infty e^{-y} \, dy \bigg) \frac {\Lambda(dp)}{p^2}\\
&= \int_0^1 (|\log p|+1) \frac{\Lambda(dp)} p < \infty .
\end{align*}
For any integer $i$ we have on the event $\sup_{t < \tau_{2^i,n}} |\log N_n(t)-Y_n(t)| \le \eps$ because of the monotonicity of $f$,
\begin{align*}
\int_0^{\tau_{2^i,n} } f(Y_n(s))\, ds &\le \sum_{j \ge i} \int_{\tau_{2^{j+1},n}}^{\tau_{2^j,n}}f(\log N_n(s)-\eps) \, ds\\
&\le \sum_{j \ge i} f(j \log 2-\eps) (\tau_{2^{j},n}-\tau_{2^{j+1},n}).
\end{align*}
From Lemma \ref{ESn} and the strong Markov property there is a $C>0$ such that
\[ E\bigg[\int_0^{\tau_{2^i,n} } f(Y_n(s))\, ds\bigg] \le C \sum_{j \ge i} f(j \log 2-\eps) \le \frac C{\log 2} \int_{(i-1)\log 2 - \eps}^\infty f(y)\, dy.\]
Choosing $i$ large enough this bound may be made arbitrarily small. In view of \eqref{SDE1} and Theorem~\ref{couplingTh} our claim follows.
\end{proof}

\begin{Rmk}\label{gimremark}{\em 
Gnedin, Iskanov, and Marynych \cite{gim11} also studied the absorption time $T_n$ by coupling with a subordinator.  
The hypothesis of Lemma 4.2 in \cite{gim11} is that $$\int_0^1 \bigg( \int_0^x \nu(y) \: dy \bigg) x^{-1} \: dx < \infty,$$
where $\nu(y) = \int_y^1 x^{-2} \: \Lambda(dx)$.  This condition is equivalent to (\ref{Gnd}).  To see this, note that
\begin{align*}
\int_0^1 (|\log x| +1) \: x^{-1} \: \Lambda(dx) &= \int_0^1 (-x\log x+x) \: x^{-2} \: \Lambda(dx) =\int_0^1 \bigg( \int_0^x(- \log y) \: dy \bigg) x^{-2} \: \Lambda(dx) \\
&= \int_0^1 (-\log y) \bigg( \int_y^1 x^{-2} \: \Lambda(dx) \bigg) dy = \int_0^1 \bigg( \int_y^1 z^{-1} \: dz \bigg) \nu(y) \: dy \\
&= \int_0^1 \bigg( \int_0^z \nu(y) \: dy \bigg) z^{-1} \: dz.
\end{align*}
}
\end{Rmk}

We now come to the proof of Theorem \ref{couplingTh}. It requires two preparatory lemmas.

\begin{Lemma}\label{binlem}
Suppose $X$ has a binomial distribution with parameters $b$ and $p$.  Then 
\begin{equation}\label{Rdef}
\log \bigg( \frac{X + 1}{b+1} \bigg) - \log p = \frac{1}{p} \bigg( \frac{X+1}{b+1} - p- \frac{1-p}{b+1} \bigg) +R,
\end{equation}
where   $$E[|R|] \leq \frac{1-p}{(b+1)p}.$$
\end{Lemma}

\begin{proof}
By the Mean Value Theorem, if $x > 0$ and $y > 0$, then there exists a positive number $z$ between $x$ and $y$ such that $\log x - \log y = z^{-1}(x - y)$.  Therefore, there exists a random variable $Z$ between $(X + 1)/(b+1)$ and $p$ such that $$\log \bigg( \frac{X + 1}{b+1} \bigg) - \log p = \frac{1}{Z} \bigg( \frac{X + 1}{b+1} - p \bigg) = \frac{1}{p} \bigg( \frac{X+1}{b+1} - p \bigg) - R',$$
where $$R' =    \bigg( \frac{1}{p}-\frac{1}{Z}   \bigg)\bigg( \frac{X + 1}{b+1} - p \bigg).$$
Clearly $R' \ge 0$. It remains to bound $E[R']$.  Because $Z$ must be between $(X +1)/(b+1)$ and $p$, we see that $|1/Z - 1/p|$ can be bounded from above by substituting $(X + 1)/(b+1)$ in place of $Z$.  We get
\begin{align*}
R' &\le \bigg( \frac{1}{p}-\frac{b+1}{X+1}   \bigg)\bigg( \frac{X + 1}{b+1} - p \bigg)= \frac {X+1}{(b+1)p} + \frac {(b+1)p}{X+1} -2.
\end{align*}
Now by (\ref{binexp}), $$E \bigg[ \frac{1}{X+1} \bigg] \leq \frac{1 }{(b+1)p}.$$
Therefore,
\begin{align*}
E[R']&\leq \frac{bp+1}{(b+1)p} -1= \frac{1-p}{(b+1)p}  .
\end{align*}
Letting $R= \frac{1-p}{(b+1)p} - R'$ proves the lemma.
\end{proof}

\begin{Lemma}\label{numblocks}
Suppose  $\Lambda((0, 1]) > 0$,  
and define $\tau_{k,n}$ as in (\ref{taukndef}). Then there exists a positive constant $C_1$,  depending on $\Lambda$ but not on $n$, such that for all $2 \le k \le n$,
\begin{equation}\label{expN1}
E \bigg[ \int_0^{\tau_{k,n}} \frac{1}{N_n(s)} \: ds \bigg] \leq \frac{C_1}{k}.
\end{equation}
\end{Lemma}

\begin{proof}
Because $\Lambda((0, 1]) > 0$, there exist positive numbers $r$ and $d$ such that $\Lambda([r, 1]) = d$.  This means that $p$-mergers with $p \geq r$ occur at rate $d$.  Let $a \in (0, r \wedge 1/2)$ and $c \in (0, d)$.  By the Law of Large Numbers, there exists a positive integer $m$ such that for $b \geq m$, whenever the coalescent has $b$ blocks, the rate of mergers that will bring the coalescent down to fewer than $(1-a)b$ blocks is at least $c$.  Let $e_b$ be the expected time, when the coalescent starts with $b$ blocks, before the number of blocks drops below $(1-a)b$.  Let $$C = \max \bigg\{\frac{1}{c}, e_2, \dots, e_m\bigg\}.$$  Then, for all $b \geq 2$, if the coalescent starts with $b$ blocks, the expected time before the number of blocks drops below $(1-a)b$ is at most $C$.  For positive integers $j$, let $$B_j = \{b \in \N: (1-a)^{-(j-1)}k \leq b < (1-a)^{-j}k\}.$$  Then the expected Lebesgue measure of $\{t: N_n(t) \in B_j\}$ is at most $C$.  Therefore,
$$E \bigg[ \int_0^{\tau_{k,n}} \frac{1}{N_n(s)} \: ds \bigg] \leq \sum_{j=1}^{\infty} \frac{C(1 - a)^{j-1}}{k} = \frac{C}{ak},$$ which implies (\ref{expN1}) with $C_1 = C/a$. 
\end{proof}

\begin{proof}[Proof of Theorem \ref{couplingTh}]
Again we construct the $\Lambda$-coalescent from the Poisson point process $\Psi$, as described at the beginning of Section \ref{sec2}.
Enumerate the points of $\Psi$ as $((t_i, p_i, u_{1,i}, \dots, u_{n,i}))_{i=1}^{\infty}$.  For each $i \in \N$, let $$X_i = \sum_{j=1}^{N_n(t_i-)} \1_{\{u_{j,i} > p_i\}},$$
which is the number of extant lines that are not included in the merger at time $t_i$.  Conditional on $p_i$ and $N_n(t_i-)$, the distribution of $X_i$ is binomial with parameters $N_n(t_i-)$ and $1 - p_i$.  Also, for all $i \in \N$, we have $N_n(t_i) = X_i + \1_{\{X_i < N_n(t_i-)\}}$.  
  Dividing both sides by $N_n(t_i-)$ and taking logs, we get $$\log N_n(t_i) - \log N_n(t_i-) = \log \bigg( \frac{X_i + \1_{\{X_i < N_n(t_i-)\}}}{N_n(t_i-)} \bigg).$$  Also, $$S(t_i) - S(t_i-) =- \log(1 - p_i).$$  It follows that for $t > 0$,
$$\log N_n(t) - (\log n - S(t)) = \sum_{i=1}^{\infty} \bigg( \log \bigg( \frac{X_i + \1_{\{X_i < N_n(t_i-)\}}}{N_n(t_i-)} \bigg) - \log(1 - p_i) \bigg) \1_{\{t_i \leq t\}}.$$ Noting
$$  \log \bigg( \frac{X_i + \1_{\{X_i < N_n(t_i-)\}}}{N_n(t_i-)} \bigg)= \log \bigg( \frac{X_i + 1}{N_n(t_i-)+1} \bigg) +  \1_{\{X_i < N_n(t_i-)\}} \log \frac{N_n(t_i-)+1}{N_n(t_i-)}$$
and letting 
\[U_n(t)= \sum_{i=1}^{\infty} \1_{\{X_i < N_n(t_i-)\}} \log\frac{N_n(t_i-)+1}{N_n(t_i-)}  \1_{\{t_i \leq t\}} ,\]
we can write
\begin{align*}
&\log N_n(t) - (\log n - S(t)) \\
&\hspace{0.5in}= \sum_{i=1}^{\infty} \bigg( \frac{1}{1 - p_i} \bigg( \frac{X_i+1}{N_n(t_i-)+1} - (1 - p_i) - \frac{p_i}{N_n(t_i-) + 1}\bigg) + R_i \bigg) \1_{\{t_i \leq t\}} +U_n(t) ,
 \end{align*}
where $R_i$ is defined as in (\ref{Rdef}), with $N_n(t_i-)$ in place of $n$, $X_i$ in place of $X$, and $1 - p_i$ in place of $p$. 

We now break this sum into pieces.  Let $\eps > 0$, and let $J = \{i \in \N: p_i \leq 1 - \eps/(4N_n(t_i-)) \}$.  For $t \geq 0$, let $$M_n(t) = \sum_{i=1}^{\infty} \frac{1}{1 - p_i}  \bigg( \frac{X_i+1}{N_n(t_i-)+1} - (1 - p_i) - \frac{p_i}{N_n(t_i-) + 1}\bigg) \1_{\{t_i \leq t \wedge T_n\}} \1_{\{i \in J\}}$$ and 
$$V_n(t) = \sum_{i=1}^{\infty}R_i \1_{\{t_i \leq t \wedge T_n\}} \1_{\{i \in J\}}.$$  The probability that $N_n(t_i) = 1$, conditional on $N_n(t_i-)$ and on the event $\{i \notin J\}$, is at least $1 - \eps/4$.  Therefore,
$$P\big( \log N_n(t) - U_n(t)-(\log n - S(t)) = M_n(t) + V_n(t) \mbox{ for all }t < T_n \big) \ge 1 - \eps/4,$$
which means that   for $k>1$
\begin{align}\label{MV1}
&P \bigg( \sup_{t\in[0,\tau_{k,n}]\cap[0,T_n)} \big|\log N_n(t) -U_n(t)- (\log n - S(t)) \big| > \frac \eps 4 \bigg) \nonumber \\
&\hspace{1in}\leq \frac{\eps}{4} + P \bigg( \sup_{t\le \tau_{k,n}} |M_n(t)| > \frac{\eps}{8} \bigg) + P \bigg( \sup_{t \le \tau_{k,n}} |V_n(t)| > \frac{\eps}{8} \bigg).
\end{align}

Conditional on $p_i$ and $N_n(t_i-)$, the random variable $$\frac{1}{1-p_i} \bigg(\frac{X_i + (1-p_i)}{N_n(t_i-) + 1} - (1 - p_i)\bigg)$$ has mean zero and variance $$\frac{N_n(t_i-) p_i}{(N_n(t_i-) + 1)^2 (1 - p_i)}.$$ 
In particular, the process $(M_n(t), t \geq 0)$ is a martingale.  Recalling the definition of $\tau_{k,n}$ from \eqref{taukndef}
and putting $l_p:= \lceil \eps/(4(1-p)) \rceil$, we get for the bracket process $\langle M_n\rangle$
\begin{align*}
\langle M_n \rangle (\tau_{k,n}) &\le \int_0^{\tau_{k,n}} \int_0^{1 - \eps/(4N_n(s))} \frac{p}{(N_n(s)+1)(1 - p)}  \: \frac{\Lambda(dp)}{p^2} \: ds \\
&\le \int_0^1 \frac{1}{1-p} \bigg( \int_0^{\tau_{k,n}} \frac{1}{N_n(s)} \1_{\{N_n(s) \ge \eps/(4(1-p))\}} \: ds \bigg) \:  \frac{ \Lambda(dp)}{p}
\\ &\le \int_0^1 \frac{1}{1-p} \bigg( \int_0^{\tau_{k,n}\wedge\tau_{l_p, n}} \frac{1}{N_n(s)}  \: ds \bigg) \: \frac{ \Lambda(dp)}p.
\end{align*}  
Combining this result with (\ref{expN1}) and using $\tau_{k,n}\wedge \tau_{l_p,n}= \tau_{k \vee l_p,n}$ we obtain
$$E\big[ \langle M_n \rangle (\tau_{k,n})\big] \leq \int_0^1 \frac{1}{1-p}  \cdot C_1\Big(\frac 1k\wedge \frac{4  (1-p)}{\eps}\Big) \cdot \frac{\Lambda(dp)}p ,$$ which is finite by (\ref{dustcond}) and goes to 0 for $k\to \infty$.  Therefore, by the $L^2$ Maximum Inequality for martingales and Markov's inequality, we get that for $k$ sufficiently large
\begin{equation}\label{MV2}
E \Big[ \sup_{t \le \tau_{k,n}} |M_n(t)|^2 \Big] \le \frac{\eps^3}{4\cdot 64} \quad \text{and} \quad P \bigg( \sup_{t\le \tau_{k,n}} |M_n(t)| > \frac{\eps}{8} \bigg)\le \frac \eps 4.
\end{equation}

We now consider the process $(V_n(t), t \geq 0)$.  
By Lemma \ref{binlem},
\begin{align}
E\Big[\sup_{t \le \tau_{k,n}} |V_n(t)|\Big] &\leq E \bigg[ \sum_{i=1}^{\infty}  |R_i| \1_{\{t_i \leq  \tau_{k,n}\}} \1_{\{i \in J\}} \bigg] \nonumber \\
&\leq E \bigg[ \int_0^{\tau_{k,n}} \int_0^{1 - \eps/(4N_n(s))} \frac{p}{(N_n(s)+1)(1 - p)}  \: \frac{\Lambda(dp)}{p^2} \: ds \bigg] \nonumber 
\end{align}
Thus as above, if $k$ is sufficiently large,
\begin{align*}
E \Big[ \sup_{t \le \tau_{k,n}} |V_n(t)| \Big]  \le \frac {\varepsilon^2}{32} \quad \text{and}\quad P\bigg(\sup_{t \le \tau_{k,n}} |V_n(t)| > \frac{\eps}8 \bigg) \le \frac{\eps} 4 .
\end{align*}
Together with \eqref{MV1} and \eqref{MV2} we arrive at
\begin{align} \label{MV3}
P\bigg( \sup_{t\in [0,\tau_{k,n}]\cap[0,T_n)} \big|\log N_n(t)-U_n(t)-(\log n-S(t))\big| > \frac \eps 4\bigg) \le \frac {3 \eps}4 .
\end{align}

Now we approximate $U_n(t)$ by $\int_0^t f(\log N_n(s)) \, ds$, uniformly for $t \le \tau_{k,n}$. Note that by (\ref{dustcond}), there are only finitely many $t_i$ such
that $t_i \leq T_n$ and $X_i < N_n(t_i-)$.  Denote these points by $s_1 < \dots < s_m$, and also set $s_0 = 0$ and $s_{m+1} = \infty$.  Note that $s_m = T_n$.  When the coalescent has $b$ blocks, the points $s_i$ appear at rate
\begin{align}
\rho(b)&= \int_0^1 \sum_{k=1}^b \binom bk p^k(1-p)^{b-k} \frac{\Lambda(dp)}{p^2} = \int_0^1 (1-(1-p)^b)\, \frac{\Lambda(dp)} {p^2} .
\label{intuition}
\end{align}
Therefore, the random variables $G_i = (s_{i+1} - s_i)\rho(N_n(s_i))$ for $0 \leq i \leq m-1$ are independent standard exponential random variables, also independent of the process $N_n(s_j), j \ge 1$.  Recalling~(\ref{defb}), we have 
$\rho(b) =bf(\log b).$
Now for $t \le T_n$
\begin{align*}
\int_0^{t} f(\log N_n(s)) \, ds &= \sum_{i=0}^{m-1} f(\log N_n(s_i))\Big((s_{i+1}-s_i)\1_{\{s_{i+1} \leq  t\}}+ (t - s_i)\1_{\{s_i <  t < s_{i+1}\}}\Big) \\
&= \sum_{i=0}^{m-1} \frac {G_i} {N_n(s_i)}\bigg(\1_{\{s_{i+1} \leq  t\}}+\frac{t - s_i}{s_{i+1}-s_i}\1_{\{s_i < t<s_{i+1}\}}\bigg).
\end{align*}
Consequently, since $U_n(t)= \sum_{i=0}^{m-1} \log\big( (N_n(s_i)+1)/N_n(s_i)\big) \1_{\{s_{i+1} \leq  t\}}$,
\begin{align*}
\int_0^{t} f(\log N_n(s)) \, ds - U_n(t) &= \sum_{i=0}^{m-1} \frac {G_i-1} {N_n(s_i)}\1_{\{s_{i+1} \leq  t\}} + \sum_{i=0}^{m-1} \frac {G_i} {N_n(s_i)} \frac{t - s_i}{s_{i+1}-s_i}\1_{\{s_i <  t< s_{i+1}\}} \\
&\qquad\mbox{}+ \sum_{i=0}^{m-1} \bigg( \frac{1}{N_n(s_i)}- \log \frac{N_n(s_i)+1}{N_n(s_i)}\bigg)\1_{\{s_{i+1} \leq  t\}}.
\end{align*}
Using that the second sum has just one non-vanishing summand, and that $x - \log(1+x) \leq x^2$ for $x \geq 0$, we have for $t \le \tau_{k,n}$
\begin{align} \label{intf}
\bigg|\int_0^{t}&f(\log N_n(s)) \, ds - U_n(t)\bigg| \\&\leq \bigg| \sum_{i=0}^{m-1} \frac{G_i-1} {N_n(s_i)}\1_{\{s_{i +1} \le t\}} \bigg| + \max_{0 \leq i \leq m-1} \frac {G_i} {N_n(s_i)}\1_{\{s_{i} <  \tau_{k,n}\}} + \sum_{i=0}^{m-1} \frac 1{ N_n(s_i)^2}\1_{\{s_{i} <  \tau_{k,n}\}}.\notag
\end{align}
We show that for $k$ sufficiently large the supremum over $t \le \tau_{k,n}$ of the right-hand side gets arbitrarily small in probability, uniformly in $n$.  To this end we deal with the three summands on the r.h.s. of \eqref{intf} in reverse order.

First we have
\[\sum_{i=0}^{m-1} \frac 1{N_n(s_i)^2}\1_{\{s_{i} <  \tau_{k,n}\}} \le \sum_{j=k}^n \frac 1{j^2} +\sum_{i=1}^{m-1} \frac 1{N_n(s_i)^2} \1_{\{ N_n(s_i) = N_n(s_{i-1})\}} \1_{\{s_{i} <  \tau_{k,n}\}}\]
and so by Lemma \ref{numblocks}
\begin{align} \label{MV5}
E\bigg[\sum_{i=0}^{m-1} \frac 1{N_n(s_i)^2}\1_{\{s_{i} <  \tau_{k,n}\}} \bigg]&\le \frac 2k +
E\bigg[ \int_0^1 \int_0^{\tau_{k,n}} \frac {N_n(s)p(1-p)^{N_n(s)-1}}{N_n(s)^2} \,ds \:\frac{\Lambda(dp)}{p^2} \bigg]\nonumber\\&\le\frac 2k+ \int_0^1E\bigg[  \int_0^{\tau_{k,n}} \frac 1{N_n(s) } \, ds \bigg] \frac{\Lambda(dp)} p \nonumber \\
&\le \frac 1k\bigg(2 + C_1 \int_0^1 \frac {\Lambda(dp)}p\bigg)  .
\end{align}

Second, since $E[G_i^2]=2$, we have for $u>0$
\begin{align*}
P\bigg( \max_{0 \leq i \leq m-1} \frac {G_i} {N_n(s_i)}\1_{\{s_{i} <  \tau_{k,n}\}} > u\bigg) &\le E\bigg[ \sum_{i=0}^{m-1} P\bigg( \frac {G_i} {N_n(s_i)}\1_{\{s_{i} <  \tau_{k,n}\}} > u \, \Big| \, N_n(s_i), i \ge 1 \bigg) \bigg]
\\& \le \frac 1 {u^2} E\bigg[ \sum_{i=0}^{m-1} \frac 2{N_n(s_i)^2} \1_{\{s_{i} <  \tau_{k,n}\}}\bigg]
\\& \le \frac 2 {u^2k}  \bigg(2+ C_1 \int_0^1 \frac {\Lambda(dp)}p \bigg) ,
\end{align*}
where we used \eqref{MV5} in the last inequality.

Third let
\[ M_n'(t) = \sum_{i=0}^{m-1} \frac {G_i-1} {N_n(s_i)}\1_{\{s_{i+1} \leq  t\}}. \]
Then $(M_n'(t), t \geq 0)$ is a martingale with 
\[E\big[  \langle M_n' \rangle (\tau_{k,n})\big]=  E\bigg[\sum_{i=0}^{m-1} \frac 1{N_n(s_i)^2}\1_{\{s_{i+1} \leq  \tau_{k,n})\}} \bigg] \le E\bigg[ \sum_{i=0}^{m-1} \frac 1{N_n(s_i)^2}\1_{\{s_{i} <  \tau_{k,n})\}} \bigg]  , \] 
and again by means of the $L_2$ Maximum inequality and \eqref{MV5}
\[ E\bigg[\sup_{t \le \tau_{k,n}}\bigg| \sum_{i=0}^{m-1} \frac{G_i-1} {N_n(s_i)}\1_{\{s_{i+1} \leq  t\}} \bigg| ^2 \bigg] \le  \frac 4k\bigg(2 + C_1 \int_0^1 \frac {\Lambda(dp)}p\bigg) . \]

Using these three estimates we obtain from \eqref{intf} that for any $\eps >0$ 
\[ P\bigg(\sup_{t \le \tau_{k,n}}\bigg| \int_0^{t} f(\log N_n(s)) \, ds - U_n(t)\bigg| > \frac \eps 4\bigg) \le \frac \eps4, \]
if $k$ is sufficiently large. Combining this bound with  \eqref{MV3} we arrive at the formula
\begin{align}\label{MV4} P\bigg( \sup_{t \in [0,\tau_{k,n}]\cap [0,T_n)} \bigg| \log N_n(t) - \int_0^t f(\log N_n(s))\, ds - (\log n -S(t))\bigg| > \frac \eps 2\bigg) \le \eps .
\end{align}

To finish the proof we define for $t\ge 0$
\begin{align*} \Delta_n(t) &:= \log N_n(t)-  \int_0^t f(\log N_n(s))\, ds - (\log n - S(t)) \\
&= \log N_n(t)-  \int_0^t f(\log N_n(s))\, ds -\bigg(Y_n(t)-\int_0^t f(Y_n(s)) ds \bigg).
\end{align*}
For fixed $t$ and $n$ we consider the event $A_{\ge} := \{t < T_n,\, t\le \tau_{k,n},\, \log N_n(t) \ge Y_n(t)\} $ and   define the random time 
$$\sigma_t:= \sup\{s\le t: \log N_n(s) \le Y_n(s)\}.$$
Then on the event $A_{\ge}$ we have 
$\log N_n(\sigma_t-) - Y_n(\sigma_t-) \le 0$ and $f(\log N_n(s))-f(Y_n(s)) \le 0$ for $s> \sigma_t$, since $f$ is decreasing. Thus, on $A_{\ge}$,
\begin{align*}
0 &\le \log N_n(t) -Y_n(t) \\ &= \log N_n(\sigma_t-)-Y_n(\sigma_t-) + \int_{\sigma_t}^t (f(\log N_n(s))-f(Y_n(s)))\, ds +\Delta_n(t)-\Delta_n(\sigma_t-)
\\
&\le \Delta_n(t) - \Delta_n(\sigma_t-) \\
&\le 2 \sup_{t \in[0,\tau_{k,n}]\cap[0,T_n)} |\Delta_n(t)|.
\end{align*}
Similarly  on $A_{\le} := \{t < T_n,\, t \le \tau_{k,n}, \, \log N_n(t) \le Y_n(t)\}$, 
$$0 \le Y_n(t)-  \log N_n(t)  \le 2 \sup_{t \in[0,\tau_{k,n}]\cap[0,T_n)} |\Delta_n(t)|.$$
Recalling (\ref{MV4}), this implies that for sufficiently large $k$,
\begin{align*}
P \bigg( \sup_{t \in[0,\tau_{k,n}]\cap[0,T_n)} \big| \log N_n(t) - Y_n(t) \big| >\eps\bigg) \le  P\bigg(\sup_{t \in[0,\tau_{k,n}]\cap[0,T_n)} |\Delta_n(t)|> \frac \eps 2\bigg) \le \eps,
\end{align*}
which was the claim.
\end{proof}

\section{Proof of Theorem 2}\label{sec3}
In this section we prove Theorem \ref{limitdist}. First we provide a lemma which gives a uniform lower bound for the probability that the block-counting process does not jump over certain intervals.
\begin{Lemma}\label{NMK}
Assume \eqref{maincond} and that $\Lambda$ is log-nonlattice.
Fix $0<\delta <1$ and $K > 1$.  Suppose $m < n \leq Km$.  Then there exist constants $C>0$ and $\alpha \in (0, 1]$, depending on $\delta$ and $K$ but not on $m$ or $n$, such that $P((1 - \delta)\alpha m \leq N_n(t) \leq \alpha m \mbox{ for some }t \geq 0) \geq C$.
\end{Lemma}

\begin{proof}
We distinguish two cases. First assume that for all $\eta >0$ we have $\Lambda((0,\eta])>0$. 
Let $\eta = 4^{-2K/\delta}$ and define $N_n'$, $ N_n''$ to be the block-counting processes belonging to the two coalescents arising by restricting $\Lambda$ to the intervals either $[0, \eta]$ or  $(\eta, 1]$, and using the same Poisson process~$\Psi$. The processes $N_n'$, $ N_n''$ are independent, therefore for any $u>0$
\begin{align*}
&P( (1-\delta)m \le N_n(t) \le m \text{ for some } t \ge 0)\\
&\hspace{.3in}\ge P( N_n''(u)=n ,\  N_n'(u) \le (1-\delta)m ,\ \sup_{t \le u} ( N_n'(t-)- N_n'(t) )\le \delta m)) \\
&\hspace{.3in}\ge P( N_n''(u)=n ,\  N_n'(u) \le (1-\delta)n/K ,\ \sup_{t \le u} ( N_n'(t-)- N_n'(t) )\le \delta n/K)) \\
&\hspace{.3in}\ge P( N_n''(u)=n )P( N_n'(u) \le (1-\delta)n/K)- P(\sup_{t \le u} (N_n'(t-)-N_n'(t) )> \delta n/K).
\end{align*}
By assumption the process $N_n'$ is non-degenerate. Thus in view of Lemma \ref{ESn} the expectation of $W_n':= \min\{ t \ge 0: N_n'(t) \le  (1-\delta)n/K\}$ is bounded by a constant $\kappa$, depending on $\delta$ and $K$ but not on $n$. Choosing $u=2\kappa$ we obtain from  Markov's Inequality
\[ P(N_n'(2\kappa) > (1-\delta)n/K) = P( W_n' \ge 2\kappa) \le \frac 1{2\kappa} E[W_n'] \le \frac 12.  \]
Moreover 
\[P(N_n''(2\kappa) =n) \ge  e^{-2\kappa \int_\eta^1 p^{-2}\, \Lambda(dp)}>0. \]

Finally, for the rate at which $N_n'$ performs at time $t$ a jump of size larger than $\delta n/K$, we obtain from \eqref{bin2} and from the choice of $\eta$ for $n \ge 4K/\delta$ the bound
\begin{align*} \int_0^\eta P(B_{N_n'(t-),p} > \delta n/K) \, \frac{\Lambda(dp)}{p^2} &\le \int_0^\eta p^{\delta n/K} 2^{N_n'(t-)} \, \frac{\Lambda(dp)}{p^2} \\&\le \eta^{\delta n/(2K)}2^n \Lambda([0,1])= 2^{-n}\Lambda([0,1]). 
\end{align*}
Therefore
\[ P(\sup_{t \le 2\kappa} (N_n'(t-)-N_n'(t) )> \delta n/K) \le 2\kappa 2^{-n} \Lambda([0,1]) \ . \]
Putting our estimates together we arrive at
\[ P( (1-\delta)m \le N_n(t) \le m \text{ for some } t \ge 0) \ge \frac 14 e^{-2\kappa \int_\eta^1 p^{-2}\, \Lambda(dp)}>0 \]
for $n$ sufficiently large and any $m$ with $m < n\le Km$.  A further lowering of this bound makes the estimate valid for all $n$. Letting $\alpha=1$ our claim follows.

For the second part of the proof let $\Lambda([0,\eta])=0$ for some $\eta >0$. Then \eqref{Gnd}  is satisfied such that we may resort to Corollary \ref{ApproxS}. Note that our log-nonlattice assumption means that the random walk $(S(i), i \in \mathbb N_0)$ is non-lattice in the usual sense.  Condition \eqref{maincond} implies $E[S(1)]<\infty$. Therefore the classical renewal theorem implies that with $\alpha$ sufficiently small there is a constant $0<C\le 1/2$ depending on $\delta$  such that for all $s\ge 0$
\[ P\Big(\exists i \in \mathbb N_0: s-\log \alpha - \frac 13 \log (1-\delta) \le S(i)\le s - \log \alpha - \frac 23 \log (1-\delta)\Big) \ge 2C ,\]
and consequently for $m< n$ (letting $s=\log n-\log m$)
\begin{align} P \Big(\exists t\ge 0:\frac 23 \log (1-\delta) +\log \alpha  m \le \log n - S(t) \le \frac 13 \log (1-\delta) +\log \alpha  m \Big)\ge 2C. \label{eq1}
\end{align}
Next, choose $k$ according to Corollary \ref{ApproxS} so that \eqref{couplingY2} holds with $\eps= \frac 14 C \wedge \frac 13 |\log(1-\delta)|$.  Let $k$ be so large that by Theorem \ref{tight}, we have
$P(\tau_{k,n}=T_n)= P( N_n(T_n-) \ge k) \le \frac 14 C$ for all $n$. Then
\begin{align} P\Big(\sup_{t \le \tau_{k,n}} |\log N_n(t)-\log n+ S(t)| \ge \frac 13 |\log (1-\delta)|\Big) \le \frac 12 C. \label{eq2}
\end{align}
In particular with $t=\tau_{k,n}$, since $k \ge N_n(\tau_{k,n})$,
\[ P\Big(\log n - S(\tau_{k,n}) \ge \log k - \frac 13 \log (1-\delta)\Big) \le \frac 12 C \]
and hence  for $n$  sufficiently large, because $m\ge n/K$, and because of a.s. monotonicity of $S$,
\[ P\Big(\forall t > \tau_{k,n} : \log n -S(t) < \log \alpha m + \frac 23 \log (1-\delta) \Big) \ge 1-\frac 12 C .\]
Intersecting  this event  with the event in \eqref{eq1} we obtain 
\[P \Big(\exists t\le \tau_{k,n}:\frac 23 \log (1-\delta) +\log \alpha   m \le \log n - S(t) \le \frac 13 \log (1-\delta) +\log \alpha  m \Big)\ge\frac 32  C. \]
Hence from \eqref{eq2} it follows for $n$ sufficiently large and $m<n\le Km$
\[ P \Big(\exists t \le \tau_{k,n}:  \log (1-\delta) +\log \alpha   m \le \log N_n(t) \le \log \alpha  m \Big) \ge C .\]
Again by suitably lowering the constant $C$ this estimate holds for all $n$, which then translates into our claim.
\end{proof}

\begin{proof}[Proof of Theorem \ref{limitdist}]
We prove this result by coupling.  Let $\eps > 0$.  It suffices to show that there exists a positive integer $n_0$ such that
if $n_0 < n_1 < n_2$, then we can construct $\Lambda$-coalescents $(\Pi_{n_1}(t), t \geq 0)$ and $(\Pi_{n_2}(t), t \geq 0)$ started with $n_1$ and $n_2$ blocks respectively such that
\begin{equation}\label{couple1}
P(N_{n_1}(T_{n_1}-) = N_{n_2}(T_{n_2}-)) > 1 - \eps.
\end{equation}

By Theorem \ref{tight}, we can choose a positive integer $\ell$ such that $P(N_n(T_n-) \leq \ell) > 1 - \eps/4$ for all~$n$.   Let $C$ be the constant from Lemma \ref{NMK} with $\delta=\eps/(4\ell)$ and with the constant $K=K_{1/2}$ from Proposition \ref{mKmprop}.  Choose a positive integer $J$ large enough that $$\bigg(1 - \frac{C^2}{4}\bigg)^J < \frac{\eps}{2}.$$  Then for $1 \leq j \leq J$, let $m_j = \lfloor n_0^{j/J} \rfloor$.  For $1 \leq j \leq J$ and $i \in \{1, 2\}$, let $A_{i,j}$ be the event that $m_j < N_{n_i}(t) \leq K m_j$ for some $t \geq 0$, and let $D_{i,j}$ be the event that $(1 - \delta)\alpha m_j \leq N_{n_i}(t) \leq \alpha m_j$ for some $t \geq 0$, with the constant $\alpha$ as in Lemma~\ref{NMK}.  It follows from Proposition \ref{mKmprop} and Lemma~\ref{NMK} that for $1 \leq j \leq J$ and $i \in \{1, 2\}$, we have
\begin{equation}\label{PBijprelim}
P(D_{i,j}) \geq P(D_{i,j}\cap A_{i,j})=P(A_{i,j}) P(D_{i,j}|A_{i,j}) \geq \frac{1}{2} C.
\end{equation}
We will need to establish that a similar inequality holds when we condition on the events $D_{i,k}$ for $k > j$.  To this end, let $U_{i,J} = 0$ for $i \in \{1, 2\}$, and for $1 \leq j \leq J-1$ and $i \in \{1, 2\}$, define the stopping time $U_{i,j} = \inf\{t\ge 0: N_{n_i}(t) \leq \alpha m_{j+1}\}$.  For $1 \leq j \leq J$ and $i \in \{1, 2\}$, let $G_{i,j} = \{N_{n_i}(U_{i,j}) > m_j\}$.  Let $(\mathcal{F}_{i}(t), t \geq 0)$ be the natural filtration associated with the process $(\Pi_{n_i}(t), t \geq 0)$.  With $N_{n_i}(U_{i,j})$ figuring as the new starting point, the reasoning leading to (\ref{PBijprelim}) implies that for $1 \leq j \leq J$ and $i \in \{1, 2\}$, we have, on the event $G_{i,j}$,
\begin{equation}\label{PDij}
P(D_{i,j}|\mathcal{F}_i(U_{i,j})) \geq \frac{1}{2} C \text{ a.s.}
\end{equation}
Because $m_{j+1}/m_j \rightarrow \infty$ as $n_0 \rightarrow \infty$, it follows from Proposition \ref{mKmprop} that
\begin{equation}\label{PGij}
\lim_{n_0 \rightarrow \infty} P(G_{i,j}) = 1.
\end{equation}
Since $D_{i,k} \in \mathcal{F}_i(U_{i,j})$ for $1 \leq j < k \leq J$ and $i \in \{1, 2\}$, the results (\ref{PDij}) and (\ref{PGij}) imply that if the processes $(\Pi_{n_1}(t), t \geq 0)$ and $(\Pi_{n_2}(t), t \geq 0)$ are independent, then
\begin{equation}\label{probcouple}
\limsup_{n_0 \rightarrow \infty} P \bigg( \bigcup_{j=1}^J (D_{1,j} \cap D_{2,j}) \bigg) \geq 1 - \bigg(1 - \frac{C^2}{4}\bigg)^J > 1 - \frac{\eps}{2}.
\end{equation}

We now couple the processes $(\Pi_{n_1}(t), t \geq 0)$ and $(\Pi_{n_2}(t), t \geq 0)$.  We allow the two processes to evolve independently until the times $U_{1,J-1}$ and $U_{2,J-1}$ respectively.  If $D_{1,J} \cap D_{2,J}$ occurs, then we stop.  Otherwise, we allow the processes to continue to evolve independently until the times $U_{1,J-2}$ and $U_{2,J-2}$ respectively.  Then we stop if $D_{1,J-1} \cap D_{2,J-1}$ occurs, and otherwise continue as before.  According to (\ref{probcouple}), with probability at least $1 - \eps/2$, we will eventually come to a value of $j$ such that $D_{1,j} \cap D_{2,j}$ occurs.  In that case, the independent constructions will be stopped at the times $U_{1,j-1}$ and $U_{2,j-1}$ respectively, at which times both processes will have between $(1 - \delta)\alpha m_j$ and $\alpha m_j$ blocks.  

We now suppose the independent constructions are stopped at the times $U_{1,j-1}$ and $U_{2,j-1}$.
Set $n_1' = N_{n_1}(U_{1,j-1})$ and $n_2' = N_{n_2}(U_{2,j-1})$.  Without loss of generality, assume $n_1' < n_2'$.  Let $B_{1,1}, \dots, B_{1,n_1'}$ and $B_{2,1}, \dots, B_{2,n_2'}$ denote the blocks of the partitions $\Pi_{n_1}(U_{1,j-1})$ and $\Pi_{n_2}(U_{2,j-1})$ respectively.  We now construct $(\Pi_{n_1}(U_{1,j-1} + t), t \geq 0)$ and $(\Pi_{n_2}(U_{2,j-1}+ t), t \geq 0)$ from the same Poisson point process $\Psi$, as described at the beginning of Section \ref{sec2}.  This means both processes will have $p$-mergers at the same times, and the number of blocks in $\Pi_{n_2}(U_{2,j-1} + t)$ that contain integers from one or more of the blocks $B_{2,1}, \dots, B_{2,n_1'}$ will equal $N_{n_1}(U_{1,j-1} + t)$.  Recall that $T_{n_2}$ is the time of the last merger in $(\Pi_{n_2}(t), t \geq 0)$.  Unless one or more blocks of $\Pi_{n_2}(T_{n_2}-)$ contains only integers from the blocks $B_{2,n_1'+1}, \dots, B_{2,n_2'}$, we will have $N_{n_1}(T_{n_1}-) = N_{n_2}(T_{n_2}-)$.  By the exchangeability of the coalescent dynamics, conditional on $n_1'$ and $n_2'$, the probability that a particular block of $\Pi_{n_2}(T_{n_2}-)$ contains only integers from the blocks $B_{2,n_1'+1}, \dots, B_{2,n_2'}$ is at most $(n_2' - n_1')/n_2'$, which is at most $\delta$ because we are assuming that $D_{1,j} \cap D_{2,j}$ occurs.  Therefore, recalling that $\ell$ was chosen so that $P(N_{n_2}(T_{n_2}-) > \ell) < \eps/4$, we have $$P(N_{n_1}(T_{n_1}-) \neq N_{n_2}(T_{n_2}-)) \leq \frac{\eps}{2} + \frac{\eps}{4} + \ell \delta = \eps,$$ which implies (\ref{couple1}).
\end{proof}

\section{Non-convergence for Eldon-Wakeley coalescents} \label{Eldon}

To provide an example where the distribution of the size of the last merger does {\em not} converge as $n\to \infty$, we now focus on the class of coalescents proposed in \cite{elwa}  and thus assume that the measure $\Lambda$ is concentrated in one point $p \neq 0,1$. Because of  Theorem~\ref{tight}, for such coalescents the size of the last merger is tight. We claim that still $L_n$ does not converge in distribution as $n\to \infty$.  There are obvious relations to non-convergence and periodicity phenomena in the  so-called leader election, see e.g. Gr\"ubel and Hagemann \cite{GH} and references therein.  

For notational convenience  we restrict ourselves to the case $\Lambda= p^2 \delta_p$ and $p=e^{-1}$.  Then the points of 
the Possion point process 
$\Psi$ are of the form $(\sigma_i,p,u_1, \ldots,u_n)$, $i=1,2,\ldots$, where the numbers $0 < \sigma_1 < \sigma_2 <\cdots$ form a standard Poisson point process on $\mathbb R_+$. Define $\tau_{k,n}$ as in (\ref{taukndef}).

We shall argue by contradiction, so let us assume that $L_n$ does converge in distribution. Then, as shown in Theorem \ref{timerev}, the sequence of time-reversed Markov chains converges as $n \to \infty$ in distribution to a limiting Markov chain. This  implies
\begin{align} \label{levelk} \forall \eps >0\ \exists k>0  :  N_n(\tau_{k,n}) \stackrel{d}\to N_{\infty,k} \text{ with } P(N_{\infty,k}\ge 2) \ge 1-\eps .
\end{align}
Together with $N_n$ we consider a process $\overline N_n \ge N_n$ defined inductively as follows: $\overline N_n(0)=N_n(0)$ and at times $\sigma_i$ the random number $\overline N_n(\sigma_i)$ is thinned according to $p$ and afterwards is increased by one. Thinking of $N_n$ and $\overline N_n$ as numbers of lines, the difference between both processes only arises, when by a thinning no line of $N_n$ is affected. Then $N_n$ does not change its value but $\overline N_n$ increases by 1. Given $N_n(t)=m$ this takes place with probability $q^{m}$ with $q=1-p$. This may occur several times, and, as long as $N_n$ stays at level $m$, the expected increase of $\overline N_n$ is bounded from above by $q^m/(1-q^m)\le q^m/p$. Therefore, given $\eps >0$ there is a $k$ such that
\[ E[\overline N_n(\tau_{k,n})-N_n(\tau_{k,n}) ] \le \sum_{m\ge k} \frac {q^m}p= \frac{q^k}{p^2} \le \eps \text{ and } P(\overline N_n(\tau_{k,n})=N_n(\tau_{k,n}) )\ge 1-\eps. \]
Combined with \eqref{levelk} we obtain that also for $\overline N_n$ the size of the first jump to 1 converges in distribution with $n \to \infty$.

Now consider a representation $\overline N_n= U_n+V_n$ with random variables $U_n(0)$ and $V_n(0)$ to be specified below, where at the times $\sigma_i$ both $U_n$ and $V_n$ are thinned independently according to $p$ and then $V_n$ is enlarged by 1. Note that for independent  $U_n(0)$ and $V_n(0)$ the Markov chains $U_n$ and $V_n$ are independent as well.  Also $U_n$ converges a.s. to zero, whereas $V_n$ is an aperiodic, irreducible chain, which is positive recurrent in view of $E[V_{n}(\sigma_{m+1})-V_n(\sigma_m)\mid V_n(\sigma_m)]= 1-p V_n(\sigma_m)$ a.s. Let $\pi$ be its stationary distribution.

Let us study the case $\overline N^\lambda=U^\lambda+V$ with independent Markov chains $U^\lambda$ and $V$, both with the dynamics described above, where now $U^\lambda(0)$ is Poisson($e^\lambda$)-distributed with $\lambda \in \mathbb R$ and $V(0)$ has the distribution $\pi$.  Since $p=e^{-1}$,  the random variable $U^\lambda(\sigma_m)$ is Poisson($e^{\lambda-m}$)-distributed.  Let $\rho = \inf\{t: {\bar N}^{\lambda}(t) = 1\}$ and $\rho' = \inf\{t: U^\lambda(t) = 0\}$. Note that $\rho'\le \rho$.

We now focus on the event $\{\overline N^\lambda(\rho-)=2\}$. It  
can occur in two different ways, either $ \rho' = \rho$ or $\rho'<\rho$. The first instance takes place if and only if for some $m \ge 0$ we have $U^\lambda(\sigma_m)=1$, $U^\lambda(\sigma_{m+1})=0$, and $V(\sigma_m)=V(\sigma_{m+1})=1$. By independence this event has probability
\[ \pi(1)e^{-1}\sum_{m=0}^\infty e^{-e^{\lambda-m}} e^{\lambda-m} e^{-1} .\]
In case of the event $\{\rho'<\rho\}$ we have $V(\rho') \ge 2$ and $V(\rho-)=2$.  This will occur if and only if, defining $h$ so that $\rho' = \sigma_h$, we have for some $\ell > h$ that $V(\sigma_i) \geq 2$ for $i = h, h+1, \dots, \ell-2$, $V(\sigma_{\ell-1}) = 2$, and $V(\sigma_{\ell}) = 1$.  By applying the strong Markov property at time $\sigma_h$ and using the independence of the two chains, we see that, letting $\sigma_0 = 0$, the probability that this occurs is
\[\alpha:=  P(V(\sigma_0), \ldots,V(\sigma_{\ell-2}) \ge 2, V(\sigma_{\ell-1})=2, V(\sigma_\ell)=1 \text{ for some } \ell \ge 1) .\]
Replacing $\lambda$ by $\lambda+n$ and letting $n\to \infty$ we obtain
\[ \lim_{n \to \infty} P(\overline N^{\lambda+n}(\rho-)=2)= \alpha + \pi(1)e^{-2} f(\lambda)\quad \text{ with }\quad f(\lambda):= \sum_{m=-\infty}^\infty  e^{-e^{\lambda-m}} e^{\lambda-m}.\]
The function $f$ is smooth with period 1. By our assumption that $L_n$ converges in distribution as $n \to \infty$, the function $f$   does  not depend on $\lambda$. To get a contradiction we compute its Fourier coefficients. They are given by
\[ \hat f(k) = \int_{-\infty}^\infty e^{-e^{\lambda}} e^{\lambda}e^{-2\pi i k \lambda} \, d\lambda= E[e^{-2\pi i k G}], \]
where the distribution of  $G$ is standard Gumbel. The characteristic function of the standard Gumbel distribution is equal to $\varphi(t)=\Gamma(1- it)$, $t \in \mathbb R$. Also the gamma function is known to possess no zeros in the complex plane, thus none of the Fourier coefficients of $f$ vanishes. Therefore $f$ is non-constant, and we arrive at the promised contradiction.

\section{Proof of Theorem 3}\label{sec4}

Our proof of Theorem \ref{nottight} relies on an overshoot estimate for subordinators.  The Renewal Theorem for subordinators (see, for example, Corollary 5.3 in \cite{kyp14}) implies that if $(S(t), t \geq 0)$ is a subordinator and $E[S(1)] = \infty$, then for all $y > 0$, $$\lim_{x \rightarrow \infty} P(S(t) \in [x, x+y] \mbox{ for some }t) = 0.$$  To prove Theorem \ref{nottight}, we will need to establish a version of this result which holds for processes that can be obtained by adding a small state-dependent negative drift to a subordinator.  

\begin{Prop}\label{subwithdrift}
Let $(S_t, t \geq 0)$ be a subordinator with $E[S_1] = \infty$.  Let $g: \R \rightarrow \R^+$ be a nonincreasing function such that 
\begin{equation}\label{limf}
\lim_{x \rightarrow \infty} g(x) = 0.
\end{equation}
For all $z > 0$, define the process $(Y^z_t)_{t \geq 0}$ to be the solution to the SDE
\begin{equation}\label{SDE}
Y_t^z = z - \bigg( S_t - \int_0^t g(Y_s^z) \: ds \bigg).
\end{equation}
For all $y \in \R$, let $\tau_y^z = \inf\{t \geq 0: Y_t^z \leq y\}$.  Then for all real numbers $K > 0$, we have
\begin{equation}\label{limhitK}
\lim_{z \rightarrow \infty} P(Y_{\tau^z_K}^z \in [-K, K]) = 0.
\end{equation}
\end{Prop}
Equation \eqref{limhitK} says that for any bounded interval the probability that $Y^z$ jumps over the interval $[-K,K]$ tends to one as the starting point $z \to \infty$.

\begin{proof}
We will prove this result by following some of the ideas from \cite{durrett} in the proof of Blackwell's Renewal Theorem in the infinite mean case.  Let $\beta^z_K = P(Y_{\tau_K^z}^z \in [-K, K])$, and let
\begin{equation}\label{betadef}
\beta_K = \limsup_{z \rightarrow \infty} \beta_K^z.
\end{equation}
Seeking a contradiction, suppose $\beta_K > 0$ for some $K$.  Because $\beta_K$ is a nondecreasing function of $K$, it suffices to obtain a contradiction when $K$ is chosen to be a sufficiently large positive integer.  We will choose $K$ to be large enough to satisfy the following four conditions:
\begin{enumerate}
\item We require $g(K) < K$, which is true for sufficiently large $K$ by (\ref{limf}).

\item We require
\begin{equation}\label{Smixing}
P(S_t \in (2(\ell - 1)K, 2 \ell K] \mbox{ for some }t \geq 0) > 0
\end{equation}
for all positive integers $\ell$.  Note that (\ref{Smixing}) may fail for small values of $K$, in particular when $S_1$ has a lattice distribution, but will hold for sufficiently large $K$.  

\item We require
\begin{equation}\label{supcond}
P \bigg( \sup_{t \geq 0} \, (g(K) t - S_t) > 1 \bigg) < \frac{1}{2}.
\end{equation}
Note that this holds for sufficiently large $K$ in view of (\ref{limf}) and the fact that $t^{-1} S_t \rightarrow \infty$ as $t \rightarrow \infty$ by the Law of Large Numbers for subordinators.

\item Let 
\begin{equation}\label{alphadef}
\alpha_K = E[ \inf\{t \ge 0: S_t - g(K) t \geq 2\}],
\end{equation}
which tends to a finite limit as $K \rightarrow \infty$ by (\ref{limf}).
We require
\begin{equation}\label{alphabeta}
\frac{2 \alpha_K (8K + 1) g(K)}{K} \leq \frac{\beta_K}{3}.
\end{equation}
If $\beta_K > 0$ for some $K$, then this condition holds for sufficiently large $K$ by (\ref{limf}) and the fact that $\beta_K$ is a nondecreasing function of $K$.
\end{enumerate}

Because (\ref{limhitK}) does not depend on the behavior of the process after time $\tau_K^z$, we may consider instead the processes $(Z^z_t)_{t \geq 0}$, defined as the solution to the SDE
\begin{equation}\label{SDEY}
Z_t^z = z - \bigg( S_t - \int_0^{t \wedge \tau^z_K} g(Z_s^z) \: ds \bigg).
\end{equation}
The processes $Z^z$ and $Y^z$ are the same until time $\tau^z_K$, which implies that
$$\beta_K^z = P(Y_{\tau_K}^z \in [-K, K]) = P(Z_{\tau_K}^z \in [-K, K]).$$
However, after time $\tau_z^K$ the process $Z^z$ is no longer affected by the drift term involving $g$.  Because $g$ is nonincreasing, we have
$Z^z_t \le  z - S_t + g(K) t$  for all $t \geq 0$.
Therefore, (\ref{supcond}) implies that
\begin{equation}\label{Yincrease}
P \bigg( \sup_{t \geq 0} Z_t^z  > z + 1 \bigg) < \frac{1}{2}.
\end{equation}

Let $U^z$ denote the potential measure associated with the process $Z^z$, meaning that $$U^z(A) = \int_0^{\infty} P(Z^z_t \in A) \: dt$$ for all Borel subsets $A$ of $\R$.  Suppose $z > K$, and $n > K$ is a positive integer.  If the process $Z^z$ enters the interval $(n - 1, n]$, then it drops below $n - 2$ after a time whose expectation is at most $\alpha_K$, and then by (\ref{Yincrease}) and the strong Markov property, the probability that the process $Z^z$ never returns to $(n-1, n]$ is at least $1/2$.  It follows that
\begin{equation}\label{Uzbound}
U_z((n - 1, n]) \leq 2 \alpha_K.
\end{equation}
Let $0 < H_1 < H_2 < \dots$ denote the points of a rate one Poisson process, independent of $(S_t)_{t \geq 0}$.  Note that the process $(Z^z_{H_n})_{n=1}^{\infty}$ has the same potential measure as $(Z^z_t)_{t \geq 0}$, in the sense that for all Borel subsets $A$ of $\R$, $$U^z(A) = \sum_{n=1}^{\infty} P(Z^z_{H_n} \in A).$$ 

We can choose an increasing sequence $(z_m)_{m=1}^{\infty}$ tending to infinity such that
\begin{equation}\label{betalim}
\lim_{m \rightarrow \infty} \beta_K^{z_m} = \beta_K.
\end{equation}
It follows from \eqref{SDEY} and the monotonicity of $g$ that 
\begin{equation}\label{zbounds}
z_m - S_{H_1} \le Z_{H_1}^{z_m} \le z_m + g(z_{m}-S_{H_1})H_1.
\end{equation}
Let $\eps > 0$.  Choose a positive integer $L$ large enough that $P(S_{H_1} \ge 2LK) < \eps$. By  \eqref{limf}   we can 
choose a positive integer $m_0$ large enough that for all $m \ge m_0$ 
$$P(g(z_{m}-S_{H_1}) H_1 \ge 2K) < \eps.$$
This together with \eqref{zbounds} implies for all 
$$P(z_m - 2LK \leq Z_{H_1}^{z_m} \leq z_m + 2K) \geq 1 - 2 \eps.$$
For the following  we also require that $z_{m_0}-2LK > K$.

 Let $\mu^z$ denote the distribution of $Z^z_{H_1}$.  By applying the strong Markov property at time $H_1$, 
we get for $m \geq m_0$,
\begin{equation}\label{betasplit}
\beta_K^{z_m} \leq \sum_{\ell = 0}^L \int_{[z_m - 2 \ell K, z_m - 2 (\ell - 1) K)} \beta_K^x \: \mu^{z_m}(dx) + 2 \eps.
\end{equation}
Write
\begin{equation}\label{amldef}
a_{m,\ell} = \int_{[z_m - 2 \ell K, z_m - 2 (\ell - 1) K)} \beta_K^x \: \mu^{z_m}(dx).
\end{equation}
It follows from (\ref{betalim}) and (\ref{betasplit}) that
\begin{equation}\label{amlbound}
\beta_K - 2 \eps \leq \liminf_{m \rightarrow \infty} \sum_{\ell = 0}^L a_{m,\ell} \leq \limsup_{m \rightarrow \infty} \sum_{\ell=0}^L a_{m,\ell} \leq \beta_K.
\end{equation}
By (\ref{limf}), for all $\ell \in \{0, 1, \dots, L\}$ we have
\begin{equation}\label{XtoS}
\lim_{m \rightarrow \infty} P\big(Z_{H_1}^{z_m} \in [z_m - 2 \ell K, z_m - 2(\ell-1)K)\big) = P\big(S_{H_1} \in (2(\ell - 1)K, 2 \ell K]\big).
\end{equation}
It follows from (\ref{betadef}) and (\ref{XtoS}) that for $\ell \in \{0, 1, \dots, L\}$, we have $$\limsup_{m \rightarrow \infty}  a_{m, \ell} \leq \beta_K P\big(S_{H_1} \in (2(\ell - 1)K, 2 \ell K]\big),$$ and then (\ref{amlbound}) yields
$$\liminf_{m \rightarrow \infty} a_{m,\ell} \geq \beta_K P\big(S_{H_1} \in (2(\ell - 1)K, 2 \ell K]\big) - 2 \eps.$$  By taking $\eps \rightarrow 0$, we see that for any fixed nonnegative integer $\ell$, we have
\begin{equation}\label{amlim}
\lim_{m \rightarrow \infty} a_{m,\ell} = \beta_K P\big(S_{H_1} \in (2(\ell - 1)K, 2 \ell K] \big).
\end{equation}
Now we also see from (\ref{amldef}) and (\ref{XtoS}) that $$\liminf_{m \rightarrow \infty} a_{m,\ell} \leq \bigg( \liminf_{m \rightarrow \infty} \sup_{x \in [z_m - 2 \ell K, z_m - 2(\ell - 1)K)} \beta_K^x \bigg) P\big(S_{H_1} \in (2(\ell - 1)K, 2 \ell K] \big).$$  In view of (\ref{Smixing}) and (\ref{amlim}), it follows that for all $\ell \in \{1, \dots, L\}$ and therefore for all positive integers $\ell$, we have
\begin{equation}\label{betax}
\liminf_{m \rightarrow \infty} \sup_{x \in [z_m - 2 \ell K, z_m - 2(\ell - 1)K)} \beta_K^x = \beta_K.
\end{equation}

Fix a positive integer $M$.  
By (\ref{betalim}) and (\ref{betax}), we can choose $m$ sufficiently large that $\beta_K^{z_m} > 2\beta_K/3$ and for $\ell \in \{1, \dots, 3M\}$, there exists a point $x_{\ell} \in [z_m - 2 \ell K, z_m - 2(\ell - 1)K)$ such that $\beta_K^{x_{\ell}} > 2 \beta_K/3$.  Set $x_0 = z_m$.  We now consider the processes $Z^{x_0}, Z^{x_3}, Z^{x_6}, \dots, Z^{x_{3M}}$, which satisfy the stochastic differential equation (\ref{SDEY}) with the same driving subordinator but different initial values.  For $1 \leq \ell \leq M$, we have
\begin{equation}\label{startY}
4K \leq Z_0^{x_{3(\ell-1)}} - Z_0^{x_{3\ell}} \leq 8K.
\end{equation}
Because $g$ is nonincreasing, the processes $Z^{x_{3(\ell - 1)}}$ and $Z^{x_{3\ell}}$ get closer together over time but do not cross, which means 
\begin{align}\label{distY}
0 \leq Z_t^{x_{3(\ell-1)}} - Z_t^{x_{3\ell}} \leq 8K
\end{align}
 for all $t \in [0, \tau_K^{x_{3\ell}}]$.  Thus,
\begin{align*}
&\int_0^{\tau_K^{x_{3\ell}}} |g(Z_t^{x_{3\ell}}) - g(Z_t^{x_{3(\ell-1)}})| \: dt \leq \sum_{n=0}^{\infty} \int_0^{\tau_K^{x_{3\ell}}} |g(Z_t^{x_{3\ell}}) - g(Z_t^{x_{3(\ell-1)}})| \1_{\{Z_t^{x_{3 \ell}} \in (K+n, K+n+1]\}} \: dt \\ 
&\hspace{1.35in}\leq  \sum_{n=0}^{\infty} \int_0^{\tau_K^{x_{3 \ell}}} |g(K+n) - g(K+n+1+8K)| \1_{\{Z_t^{x_{3 \ell}} \in (K+n, K+n+1]\}} \: dt.
\end{align*}
In view of (\ref{Uzbound}), we get a telescoping sum, and
\begin{align}\label{expfint}
E \bigg[ \int_0^{\tau_K^{x_{3\ell}}} |g(Z_t^{x_{3(\ell-1)}}) - g(Z_t^{x_{3\ell}})| \: dt \bigg] &\leq 2 \alpha_K \sum_{n=0}^{\infty} \big( g(K+n) - g(K+n+1+8K) \big) \nonumber \\
&\leq 2 \alpha_K \sum_{n=0}^{8K} g(K+n) \nonumber \\
&\leq 2 \alpha_K (8K+1) g(K).
\end{align}
Let $D_{\ell}$ be the event that $$\int_0^{\tau_K^{x_{3\ell}}} |g(Z_t^{x_{3(\ell-1)}}) - g(Z_t^{x_{3\ell}})| \: dt \leq K.$$
By Markov's Inequality and (\ref{expfint}),
\begin{equation}\label{dlcomp}
P(D_{\ell}^c) \leq \frac{2 \alpha_K (8K+1)g(K)}{K}.
\end{equation}
It follows from (\ref{startY}) that on the event $D_{\ell}$, we have $Z_t^{x_{3(\ell-1)}} - Z_t^{x_{3 \ell}} \geq 3K$ for all $t \in [0, \tau_k^{x_{3 \ell}}]$.  Furthermore, after time $\tau_K^{x_{3 \ell}}$, the process $Z^{x_{3 \ell}}$ is no longer affected by the drift term involving $g$, and thus it decreases at least as fast as $Z^{x_{3(\ell - 1)}}$.  It follows that on $D_{\ell}$, we have $Z_t^{x_{3(\ell-1)}} - Z_t^{x_{3 \ell}} \geq 3K$ for all $t \geq 0$, and thus the process $Z^{x_{3 \ell}}$ can not be in the interval $[-(K+1), K]$ at the same time as $Z^{x_{3 (\ell - 1)}}$ or any other process $Z^{x_{3j}}$ with $j < \ell$.  Let
\begin{displaymath}
I_{\ell} = \left\{
\begin{array}{ll}
\{t\ge 0: -(K+1) \leq Z_t^{x_{3 \ell}} \leq K \mbox{ and }\tau_K^{x_{3 \ell}} \leq t \leq \tau_K^{x_{3 \ell}} + 1\} & \mbox{ on }D_{\ell}  \\
\emptyset & \mbox{ on }D_{\ell}^c
\end{array} \right.
\end{displaymath}
The discussion above implies that the sets $I_{\ell}$ are disjoint.
Let $$\kappa = E[1 \wedge \inf\{t: S_t > 1\}].$$  Given the event $D_{\ell} \cap \{Z_{\tau_K^{x_{3 \ell}}} \in [-K, K]\}$, the expected Lebesgue measure of $I_{\ell}$ is at least $\kappa$.  Therefore, using (\ref{dlcomp}) and the fact that $\beta_K^{x_{3 \ell}} > 2 \beta_K/3$ followed by (\ref{alphabeta}), we get
$$E \bigg[ \int_0^{\infty} \1_{\{t \in I_{\ell}\}} \: dt  \bigg] \geq \kappa \bigg( \frac{2\beta_K}{3} - \frac{2 \alpha_K (8K + 1) g(K)}{K} \bigg) \geq \frac{\kappa \beta_K}{3}.$$  On the event that $Z_{\tau_K^{x_{3 \ell}}}^{x_{3 \ell}} \in [-K, K]$, because of \eqref{distY}, we have $Z^{z_m}_{\tau_K^{x_{3 \ell}}} \leq (8 \ell + 1)K$.  During the next time unit, the process $Z^{z_m}$ can increase by at most $g(K)$, so if $t \in I_{\ell}$, then using that $g(K) < K$, we get $$Z^{z_m}_t \leq (8 \ell + 1)K + g(K) \leq 10 \ell K.$$  
We next note that if $t \in I_{\ell}$ then $Z_t^{z_m} \geq K$ because $Z_t^{z_m} - Z_t^{x_{3 \ell}} \geq 3K$ as described above.
It follows that $$U^{z_m}([K, 10 \ell K]) = E \bigg[ \int_0^{\infty} \1_{\{K \leq Z_t^{z_m} \leq 10 \ell K\}} \: dt \bigg] \geq \sum_{j=1}^{\ell} E \bigg[ \int_0^{\infty} \1_{\{t \in I_j\}} \: dt  \bigg] \geq \frac{\kappa \beta_K \ell}{3},$$
and therefore if $y \geq 10K$, then
\begin{equation}\label{UzK}
U^{z_m}([K, y)) \geq \frac{\kappa \beta_K y}{60K}.
\end{equation}
Because the process $(Z_{H_n}^{z_m})_{n=0}^{\infty}$ is decreasing after it drops below the level $K$, it can only jump below zero one time.  In particular, the expected number of times the process jumps below zero is bounded above by one.  Therefore, letting $\nu_x$ denote the conditional distribution of Ê$  Z_{H_n}^{z_m} -Z_{H_{n+1}}^{z_m}$ given $Z_{H_n} = x$, we have $$1 \geq \int_K^{\infty} \nu_x([x, \infty)) \: U^{z_m}(dx) \geq \int_K^{3M} \nu_x([x, \infty)) \: U^{z_m}(dx).$$ 
Let $\mu$ denote the distribution of the random variable $S_{H_1} - H_1 g(K)$.  Because $g$ is decreasing, we have $\nu_x([x, \infty)) \geq \mu([x, \infty))$ for all $x \geq K$.  Therefore, $$1 \geq \int_K^{3M} \mu([x, \infty)) \: U^{z_m}(dx) = \int_K^{\infty} \int_K^{y \wedge 3M} U^{z_m}(dx) \: \mu(dy) \geq \int_K^{3M} U^{z_m}([K, y)) \: \mu(dy).$$
Combining this result with (\ref{UzK}) gives
$$1 \geq \frac{\kappa \beta_K}{60K} \int_{10K}^{3M} y \, \mu(dy).$$
Because $E[S_1] = \infty$, we have $E[S_{H_1} - H_1 g(K)] = \infty$, so the right-hand side is bigger than one for sufficiently large positive integers $M$, a contradiction.
\end{proof}

\begin{proof}[Proof of Theorem \ref{nottight}]
Let $K \geq 2$ be a positive integer.  If $2 \leq N_n(T_n-) \leq K$ and the event in (\ref{couplingY1}) holds, then
\begin{equation}\label{subhit}
-L + \log 2 \leq \log n - \bigg( S(T_n-) - \int_0^{T_n} f(Y_n(s)) \: ds \bigg) \leq L + \log K,
\end{equation}
with $f$ defined in \eqref{defb}, and the left inequality holds with $T_n-$ replaced by any $t \in [0, T_n)$. In particular, putting  $K':= L + \log K$, we have
\begin{equation}\label{subhit2}
-K' \leq Y_n(t)\quad  \mbox{ for all } t \in [0,T_n).
\end{equation}
The right inequality in \eqref{subhit} says that $Y_n(T_n-) \le K'$. With $z:= \log n$ we have $Y_n(t) = Y^z_t$ in the notation of Proposition \ref{subwithdrift}, hence $\tau_{K'}^z < T_n$. Thus $-K' \le Y^z_{\tau_{K'}^z}$ by \eqref{subhit2}. On the other hand  we have $Y^z_{\tau_{K'}^z} \le K'$ by definition, and consequently $Y^z_{\tau_{K'}^z} \in [-K',K']$.  Note that $E[S_1] = \infty$ by (\ref{maincond2}).  Therefore, combining \eqref{limhitK} and \eqref{couplingY1} we see that $P(N_n(T_n-) \leq K) \to 0$ as $n \to \infty$, which proves Theorem \ref{nottight}.
\end{proof}

\section{Proof of Theorems 4 and 5}

We prepare the proof of Theorem \ref{quasiinv} by a few lemmas.

\begin{Lemma} \label{tightness}
Let $i \ge 2$ and $\eps >0$. Then there is a $k >i$ such that for all $n$
\[  P( L_n=i, N_n(t) \notin [i+1,k] \text{ for all } t \ge 0)\le \eps .\]
\end{Lemma}

\begin{proof}
Without loss of generality $\Lambda(\{0\})=0$, because otherwise the coalescent comes down from infinity, and the claim is immediate.

Recall the definition of
$ \tau_{k, n}$ in \eqref{taukndef}. We have
\[   P( L_n=i, N_n(t) \notin [i+1,k] \text{ for all } t \ge 0)\le  P(N_n(\tau_{\ell,n})=i) \]
with $\ell=k+1$. As before, let $(t_m,p_m)$, $m \ge 1$, be the first two coordinates of the points of $\Psi$ in an arbitrary order. Denote $\tilde p :=p_m$ if $t_m=\tau_{\ell,n}$. Define for $\kappa=6i/\eps$  the events
\begin{align*}
A:&=  \Big\{  \exists i \ge 1: t_m \le \tau_{\ell,n}\ ,\   \frac {1}{\kappa N_n(t_m-)} <1- p_i\le   \frac {\kappa}{N_n(t_m-)} \Big\},\\
B:&= \Big \{ 1- \tilde p \le \frac 1{\kappa N_n(\tau_{\ell,n}-)} \Big\}, \\
C:&=  \Big \{ 1- \tilde p > \frac {\kappa}{N_n(\tau_{\ell,n}-)} \Big\}.
\end{align*}
Then 
\begin{align}  P( N_n(\tau_{\ell,n}) =i) \le  P(A) +  P(\{N_n(\tau_{\ell,n}) =i\}\cap B)+ P(\{N_n(\tau_{\ell,n}) =i\}\cap C). \
\label{tight1}
\end{align}
We estimate these probabilities.

First denote
\[ \sigma_j = \tau_{n/\kappa^{j},n}, \quad j=0,1,\ldots,  \]
and let $r$ be the smallest integer such that $n/\kappa^{r} \le \ell$. Then
\begin{align*} 
 P(A) &\le \sum_{j=0}^{r-1}  P\Big( \exists i: \sigma_j < t_m \le \sigma_{j+1},  \frac {1}{\kappa N_n(t_m-)} < 1- p_m \le  \frac {\kappa}{N_n(t_i-)} \Big)\\
&\le \sum_{j=0}^{r-1}  P\Big( \exists i: \sigma_j < t_m \le \sigma_{j+1}, \frac{\kappa^{j-1}}n < 1-p_i \le  \frac{\kappa^{j+2}}n \Big).
\end{align*}
From Lemma \ref{ESn} we have $ E[ \sigma_{j+1}-\sigma_j] \le C_\kappa$ for a suitable constant $C_\kappa$ depending on $\kappa$, thus
\begin{align*}
 P(A) &\le \sum_{j=0}^{r-1}  E[ \sigma_{j+1}-\sigma_j] \int_{[1-\kappa^{j+2}/n, 1- \kappa^{j-1}/n)} \frac{\Lambda(dp)}{p^2} \\
&\le 3 C_\kappa \int_{[1- \kappa^{r+1}/n,\, 1)} \frac{\Lambda(dp)}{p^2} \\
&\le 3 C_\kappa \int_{[ 1- \kappa/\ell, \,1)}  \frac{\Lambda(dp)}{p^2}.
\end{align*}
Thus, if we choose $\ell$ sufficiently large we obtain
\begin{align}
  P(A) \le \frac\eps3 .
\label{tight3}
\end{align}

Second we have on the event $B$ with $b= N_n(\tau_{\ell,n}-)$
\[   \frac{ \tilde p}{(b-i+2)(1-\tilde p)} \ge  \frac 1{b}\Big( \frac 1{1-\tilde p} -1 \Big) \ge \frac{b\kappa-1}{b}\ge \frac \kappa2 \] 
and consequently on the event $B$ with $\alpha=\sum_{j<\ell} \binom b{j-1} (1-\tilde p)^{j-1}\tilde p^{b-j+1}$
\begin{align*}  P( \{N_n(\tau_{\ell,n} ) &=i\}\cap B \mid \tilde p, N_n(\tau_{\ell,n}-)=b,B) = \frac 1\alpha \binom b{i-1} (1-\tilde p)^{i-1}\tilde p^{b-i+1}\\&= \frac{b-i+2}{i-1} \frac{1-\tilde p}{\tilde p}  P( \{N_n(\tau_{\ell,n})=i-1\}\cap B \mid \tilde p, N_n(\tau_{\ell,n}-)=b,B) \\
&\le \frac 2{\kappa (i-1)}. 
\end{align*}
Thus, since $\kappa \ge 6/\eps$
\begin{align} P( \{N_n(\tau_{\ell,n})=i\}\cap B) \le \frac\eps3 .
\label{tight4}
\end{align}

Third we have on the event $C$, again with $b= N_n(\tau_{\ell,n}-)$ and with $b \ge 2i$
\[ \frac{\tilde p}{(b-i+1)(1-\tilde p)} \le  \frac 2b  \frac 1{1-\tilde p} \le \frac 2 \kappa \]
and consequently for $\ell \ge 2i$
\begin{align*}  P( \{N_n(\tau_{\ell,n} )&=i\}\cap C \mid \tilde p, N_n(\tau_{\ell,n}-)=b,C) \\&= \frac{i}{b-i+1} \frac{\tilde p}{1-\tilde p}  P( \{N_n(\tau_{\ell,n})=i+1\}\cap C \mid \tilde p, N_n(\tau_{\ell,n}-)=b,C) \\
&\le \frac {2i}{\kappa} 
\end{align*}
implying
\begin{align}  P( \{N_n(\tau_{\ell,n})=i\}\cap C) \le \frac \eps3 
\label{tight5}
\end{align}
for $\kappa= 6i/\eps$ . Now from \eqref{tight1}, \eqref{tight3}, \eqref{tight4} and \eqref{tight5} our claim follows.
\end{proof}

Recall that  $\rho_{ij}$  denotes the rate for a jump of $N_n$ from state $i$ to $j$, and $\rho_i$ is the rate at which $N_n$ leaves $i$.  Next let for $n\in \mathbb N$
\[ \mu_i^{(n)}:= \frac 1{\rho_i} P(N_n(t)= i \text{ for some }t \ge 0). \]
Also let
\[ P_{ij}:= \frac{\rho_{ij}}{\rho_i}  , \quad 1 \le j < i,\]
be the transition probability from state $i$ to $j$ of the block-counting process of our $\Lambda$-coalescent.
\begin{Lemma}\label{quasi1}
Suppose that there are numbers $\mu_i$, $i \ge 2$, not all vanishing, such that for some increasing sequence $(n_m)_{m \ge 1}$ of natural numbers, as $m \to \infty$,
\[ \mu_i^{(n_m)} \to \mu_i \]
 for all $i\ge 2$. Then the measure $\mu=(\mu_i)_{i\ge 2} $ is $\rho$-invariant, i.e. satisfies the first condition in \eqref{quasiinva}.
\end{Lemma}

\begin{proof}
First we have for $i \ge 2$
\[  \mu_i^{(n)}\rho_{i1} = P(N_n(t)=i \text{ for some }t\ge 0) P_{i1}=P(L_n=i) \]
and therefore in the limit (along the specified sequence) by Fatou's Lemma
\[ \sum_{i\ge 2} \mu_i\rho_{i1} \le 1.\]

Second for  $2 \le i <k$
\begin{align*} P(L_n=i, N_n(t) \not\in [i+1,k] \text{ for all } t \ge 0)&=\sum_{j>k}  P(N_n(t)=j \text{ for some }t\ge 0)P_{ji}P_{i1}\\&=  \sum_{j > k} \mu_j^{(n)} \rho_{ji}P_{i1}  .
\end{align*}
Applying Lemma \ref{tightness} to the left-hand term it follows that for any $\varepsilon >0$ there is a $k$ such that for all $n$
\[ \sum_{j > k} \mu_j^{(n)} \rho_{ji} \le \varepsilon .\]
Therefore we may proceed in the equation
\[ \mu_i^{(n)}\rho_i = \sum_{j=i+1}^n \mu_j^{(n)} \rho_{ji} \]
along the given subsequence to the limit  to obtain
\begin{align}  
\mu_i\rho_i = \sum_{j=i+1}^\infty \mu_j \rho_{ji}, \quad i \ge 2 .
\label{gleichung2}
\end{align}
Thus $\mu$ is $\rho$-invariant.
\end{proof}

\begin{Lemma}\label{quasi2}
Let $\nu=(\nu_i)_{i\ge 2}$ be a measure satisfying \eqref{quasiinva} . Then for any integer $a\ge 1$ there are probability measures $\omega_a=(\omega_{i,a})_{1\le i \le a}$ on $\{1, \ldots,a\}$ such that for any $i \ge 2$ we have $\omega_{i,a} \to 0$ as $a\to \infty$, and for $1 \le i \le a$
\begin{align}\label{gleichung4}
 \nu_i= \sum_{n=i}^a \mu_i^{(n)} \omega_{n,a} .
 \end{align}
\end{Lemma}

\begin{proof} Denote for $i,j \ge 1$
\[ \tilde P_{ij} := \frac{\nu_j\rho_{ji}}{\nu_i\rho_i}= \frac{\nu_j\rho_j}{\nu_i\rho_i}P_{ji}, \]
where we set the undefined quantity $\nu_1\rho_1$ equal to 1.
Then the $\rho$-invariance and the norming of~$\nu$ (according to the second condition in \eqref{quasiinva}) implies $\sum_{j} \tilde P_{ij}=1$ for $i \ge 1$. Thus we may consider the Markov chain $(\tilde X_r)_{r = 0,1,\ldots}$ on $\mathbb N$  with initial state $\tilde X_0=1$ and transition matrix $(\tilde P_{ij})$.  We claim that it fulfils the equation
\[ \nu_j\rho_j=  P( \tilde X_r= j \text{ for some } r) , \quad j \ge 1. \]
We show this claim by induction. For $j=1$ both terms are equal to 1.
Suppose that it holds for $1\le i \le j-1$. Then
\begin{align*} 
 P( \tilde X_r= j \text{ for some } r) =  \sum_{i=1}^{j-1} \nu_i\rho_i \tilde P_{ij} = \nu_j\rho_j\sum_{i=1}^{j-1} P_{ji}=\nu_j\rho_j. 
\end{align*}
Next define for an integer $a>1$ the random times
\[ \xi_{a} := \max \{ r \ge 0:  \tilde X_r \le a \}\]
and for $1\le i < a$
\[ \eta_{ia} :=  P(\tilde X_{1}> a \mid \tilde X_0=i) .\]
Then for $a>1$ and $1=i_0<i_1<i_2 < \cdots < i_r\le a$
\begin{align*}
 P( \tilde X_0 =i_0,\tilde X_1=i_1, \tilde X_2=i_2,  \ldots, \tilde X_{r} = i_r, \xi_a=r)&=  \tilde P_{1i_1}\tilde P_{i_1i_2} \cdots \tilde P_{i_{r-1}i_r} \eta_{i_ra}\\
&= \omega_{i_r,a} P_{i_ri_{r-1}} \cdots P_{i_2i_1}P_{i_11} 
\end{align*}
with
\[ \omega_{i,a}:= \nu_i\rho_i \eta_{ia} , \quad 1 \le i < a\]
and $i_r=1$ in the case $r=0$ (then both products of transition probabilities are set to be 1).
For fixed $i$, summing over $1<i_1<i_2 < \cdots < i_r:=i \le a$ and $ r \ge 0$ we obtain the equality $ P(\tilde X_{\xi_a}=i)\eta_{ia}=\omega_{i,a}$, and thus $\sum_{1\le i \le a}\omega_{i,a}=1$. Therefore we may view the time-reversed process $Y_0=\tilde X_{\xi_a}, Y_1= \tilde X_{\xi_a-1}, \ldots, Y_{\xi_a}=\tilde X_0$ as a Markov chain on $\{1, \ldots, a\}$ with initial distribution $\omega_a$, transition probabilities $P_{ij}$ and killed after reaching 1. This process coincides in distribution with the block-counting process of our original coalescent in discrete time, now with initial distribution $\omega_a$. This gives another way to express $\nu_i$: For $1\le i < a$
\begin{align*}
\rho_i\nu_i=  P( Y_r= i \text{ for some } r\le \xi_a)= \sum_{n=i}^{a-1} \rho_i\mu_i^{(n)} \omega_{n,a},
\end{align*}
which is \eqref{gleichung4}. Also $\eta_{ia}\to 0$ for $a \to \infty$, which implies $\omega_{i,a}\to 0$. Thus the proof is finished.
\end{proof}

\begin{proof}[Proof of Theorem \ref{quasiinv}]
(i) Let $i \ge 2$. If $L_n \to \infty$ in probability, then as $n \to \infty$
\[ \mu_i^{(n)} = \frac{P(L_n=i)}{\rho_{i1}} \to 0 .\]
Now suppose that there is a  measure $\nu$ satisfying \eqref{quasiinva}. Then we may apply Lemma \ref{quasi2}. Let $\eps>0$ and $b>i$ such that $\mu_i^{(n)}\le \eps$ for $n >b$. From \eqref{gleichung4}  for $a>b$
\[ \nu_i \le \sum_{n=i}^b \frac 1{\rho_{i1}}\omega_{n,a}+ \sum_{n=b+1}^a \eps \omega_{n,a}\le \sum_{n=i}^b \frac 1{\rho_{i1}}\omega_{n,a}+  \eps .\]
 In the limit $a\to \infty$, since $\omega_{n,a}\to 0$ for fixed $n$, we obtain $\nu_i\le \eps$. Thus $\nu_i=0$ for all $i \ge 2$, which  is a contradiction. Hence there is no solution to \eqref{quasiinva}.

(ii) Now by assumption there is  an increasing sequence of natural numbers $n_m$, $m\ge 1$, such that as $m \to \infty$
\[ \mathbf \mu_i^{(n_m)} = \frac{P(L_{n_m}=i)}{\rho_{i1}} \to \alpha \frac {\pi_i}{\rho_{i1}}\]
for all $i \ge 2$ and for some $\alpha >0$. From Lemma \ref{quasi1} it follows that $\mu_i:= \pi_i/\rho_{i1}$ are the weights of a $\rho$-invariant measure $\mu$.

Now let $\nu$ be any solution of \eqref{quasiinva}.  By assumption we have $\mu_i^{(n)} \sim \frac{\mu_i}{\mu_2} \mu_2^{(n)}$ as $n \to \infty$. Therefore from Lemma \ref{quasi2} it follows by a similar argument as in the proof of (i) that, as $a \to \infty$,
\[ \nu_i = \sum_{n=i}^a \mu_i^{(n)} \omega_{n,a} \sim \sum_{n=i}^a \frac{\mu_i}{\mu_2}\mu_2^{(n)} \omega_{n,a} \sim \frac{\mu_i}{\mu_2} \sum_{n=2}^a \mu_2^{(n)} \omega_{n,a} =  \frac{\mu_i}{\mu_2} \nu_2. \]
This shows that $\nu$ is a multiple of $\mu$.

(iii)  In the remaining situation by means of a diagonal argument there  are two increasing  sequences such that $\mu_i^{(n)}$ converges along both sequences for all $i\ge 2$, but now the limiting measures are not multiples of each other. Thus another application of Lemma \ref{quasi1} gives the claim. This finishes the proof.
\end{proof}

\begin{proof}[Proof of Theorem \ref{timerev}]

Let   $0=\gamma_0 < \gamma_1 < \cdots < \gamma_{\zeta_n}=T_n $ be the jump times of $\hat N_n$   and let $\Delta_i:= \gamma_{i+1}-\gamma_{i}$ the interim times.
For the proof  it is now sufficient to show for fixed $r \ge 1$ convergence in distribution of the random vectors $(\hat N_n(0), \Delta_0, \ldots ,  \hat N_n(\gamma_r),\Delta_r)$ to the corresponding limiting distribution. The event $\{\zeta_n<r\}$ has vanishing probability as $n \to \infty$. In view of the strong Markov property of $N_n$ as $n \to \infty$ we have for $2\le i_0< i_1 < \cdots < i_r<n$
\begin{align*} 
 P(\hat N_n&(0)=i_0, \Delta_0 \in dt_0, \ldots,\hat N_n(\gamma_r)=i_r, \Delta_r\in dt_r) \\
&=  P( N_n((T_n-\gamma_r)-)= i_r, \Delta_r \in dt_r,\ldots, N_n(T_n-)=i_0, \Delta_0 \in dt_0)
\\&= \mu_{i_r}^{(n)} \rho_{i_r}\cdot e^{-\rho_{i_r}t_r}   \cdot \rho_{i_ri_{r-1}} \, dt_r \cdots e^{-\rho_{i_0}t_0}    \cdot  \rho_{i_01}\, dt_0 .
\end{align*}
Theorem \ref{quasiinv} (ii) implies
\begin{align*} 
 P(\hat N_n&(0)=i_0, \Delta_0 \in dt_0, \ldots,\hat N_n(\gamma_r)=i_r, \Delta_r\in dt_r) \\
&\to 
\mu_{i_r} \rho_{i_r}\cdot e^{-\rho_{i_r}t_r}   \cdot \rho_{i_ri_{r-1}} \, dt_r \cdots e^{-\rho_{i_0}t_0}    \cdot  \rho_{i_01}\, dt_0 .
\end{align*}
For $i<j$ define rates $\hat \rho_{ij}$ and $\hat \rho_i$    by
\[  \mu_i\hat \rho_{ij} = \mu_j  \rho_{ji}, \quad \hat \rho_i:= \sum_{j>i} \hat \rho_{ij}.\] 
Since $\mu$ is $\rho$-invariant,
\[ \hat \rho_i = \frac 1{\mu_i }\sum_{j >i}\mu_j\rho_{ji}=\rho_i. \]
With these terms  the above convergence statement transforms into
\begin{align*} 
 P(\hat N_n&(0)=i_0, \Delta_0 \in dt_0, \ldots,\hat N_n(\gamma_r)=i_r, \Delta_r\in dt_r) \\ &\mbox{}Ê\to \mu_{i_0}\rho_{i_01}\cdot e^{-\hat \rho_{i_0}t_0} \cdot \hat \rho_{i_0i_1}\, dt_0 \cdots e^{-\hat \rho_{i_{r-1}}t_{r-1}} \cdot
\hat \rho_{i_{r-1}i_r} \, dt_{r-1} \cdot e^{- \hat \rho_{i_r}t_r} \cdot \hat \rho_{i_r} \, dt_r \\
&=  P(\hat N_\infty(0)=i_0, \Delta_0 \in dt_0, \ldots,\hat N_\infty(\gamma_r)=i_r, \Delta_r\in dt_r).
\end{align*}
This is our claim.
\end{proof}

\bigskip
\noindent {\bf {\Large Acknowledgment}}

\bigskip
\noindent We are grateful to a referee for valuable remarks which helped to improve the presentation.


\begin{thebibliography}{99}
\bibitem{ad13} R. Abraham and J.-F. Delmas (2013).  A construction of a $\beta$-coalescent via the pruning of binary trees.  {\it J. Appl. Probab.} {\bf 50}, 772-790.

\bibitem{ad15} R. Abraham and J.-F. Delmas (2015).  $\beta$-coalescents and stable Galton-Watson trees.  {\it ALEA Lat. Am. J. Probab. Math. Stat.} {\bf 12}, 451-476.

\bibitem{durrett} R. Durrett (2010).  {\it Probability: Theory and Examples}.  4th ed.  Cambridge University Press.

\bibitem{elwa} B. Eldon and J. Wakeley (2006).  Coalescent processes when the distribution of offspring number among individuals is highly skewed.  {\it Genetics} {\bf 172} 2621--2633.

\bibitem{gim11} A. Gnedin, A. Iksanov, and A. Marynych (2011).  On $\Lambda$-coalescents with dust component.  {\it J. Appl. Probab.} {\bf 48}, 1133-1151.

\bibitem{gm05} C. Goldschmidt and J. B. Martin (2005).  Random recursive trees and the Bolthausen-Sznitman coalescent.  {\it Electron. J. Probab.} {\bf 10}, 718-745.

\bibitem{GH} R. Gr\"ubel and  K. Hagemann (2016).  Leader election: a Markov chain approach. {\it Mathematica Applicanda} {\bf 44}, 113-143.

\bibitem{henard} O. H\'enard (2015).  The fixation line in the $\Lambda$-coalescent. {\it Ann. Appl. Probab.}  {\bf 25}, 3007-3032.

\bibitem{kyp14} A. E. Kyprianou (2014).  {\it Fluctuations of L\'evy Processes with Applications}.  2nd ed.  Springer, Heidelberg.

\bibitem{mohle10} M. M\"ohle (2010). Asymptotic results for coalescent processes without proper frequencies
and applications to the two-parameter Poisson-Dirichlet coalescent. {\it Stochastic Process. Appl.} {\bf 120}, 2159-2173.

\bibitem{mohle14} M. M\"ohle (2014).  On hitting probabilities of beta coalescents and absorption times of coalescents that come down from infinity.  {\it ALEA Lat. Am. J. Probab. Math. Stat.} {\bf 11}, 141-159.

\bibitem{pit99} J. Pitman (1999).  Coalescents with multiple collisions. {\it Ann. Probab.} {\bf 27}, 1870-1902.

\bibitem{sagitov} S. Sagitov (1999).  The general coalescent with asynchronous mergers of ancestral lines.  {\it J. Appl. Probab.} {\bf 36}, 1116-1125.

\bibitem{sch00} J. Schweinsberg (2000).  A necessary and sufficient condition for the $\Lambda$-coalescent to come down from infinity.  {\it Electron. Comm. Probab.} {\bf 5}, 1-11.
\end{thebibliography}
\end{document}